\documentclass[12pt]{article}
\usepackage{graphicx,amssymb,color}
\usepackage[bookmarks=true]{hyperref}
\newcommand{\dsp}{\displaystyle}
\newcommand{\bd}{\begin{displaymath}}
\newcommand{\be}{\begin{equation}}
\newcommand{\ba}{\begin{array}}
\newcommand{\ed}{\end{displaymath}}
\newcommand{\ee}{\end{equation}}
\newcommand{\ea}{\end{array}}
\newcommand{\espace}{\mbox{ }}

\newcommand{\Prob}{{\rm I\hspace{-0.8mm}P}}
\newcommand{\Exp}{{\rm I\hspace{-0.8mm}E}}
\newcommand{\indicator}[1]{{\mbox{\large\bf$1$}}_{#1}}

\def\N{\mathbb{N}}
\def\Z{\mathbb{Z}}
\def\R{\mathbb{R}}

\newcommand{\eqref}[1]{(\ref{#1})}

\newtheorem{theorem}{Theorem}[section]

\newtheorem{proposition}{Proposition}[section]

\newtheorem{lemma}{Lemma}[section]
\newtheorem{corollary}{Corollary}[section]

\newtheorem{remark}{Remark}[section]

\newenvironment{proof}[2]{\espace\\{\em Proof of #1 \ref{#2}.}}{\hfill\mbox{$\square$}}
\begin{document}
\title{Supercritical behavior of asymmetric zero-range process with sitewise disorder}
\author{C. Bahadoran$^{a}$, T. Mountford$^{b}$, K. Ravishankar$^{c}$, E. Saada$^{e}$}
\date{}
\maketitle
$$ \ba{l}
^a\,\mbox{\small Laboratoire de Math\'ematiques, Universit\'e Blaise Pascal, 63177 Aubi\`ere, France} \\
\quad \mbox{\small e-mail:
bahadora@math.univ-bpclermont.fr}\\
^b\, \mbox{\small Institut de Math\'ematiques, \'Ecole Polytechnique F\'ed\'erale, 
Lausanne, Switzerland
} \\
\quad \mbox{\small e-mail:
thomas.mountford@epfl.ch}\\
^c\, \mbox{\small NYU-ECNU Institute of Mathematical Sciences, NYU Shanghai, China}\\
\quad \mbox{\small e-mail:
 kr26@nyu.edu}\\
^e\, \mbox{\small CNRS, UMR 8145, MAP5,
Universit\'e Paris Descartes,
Sorbonne Paris Cit\'e, France}\\
\quad \mbox{\small e-mail:
Ellen.Saada@mi.parisdescartes.fr}\\ \\
\ea
$$
\begin{abstract}
We establish necessary and sufficient conditions for
weak convergence to the upper invariant measure for  one-dimensional 
asymmetric nearest-neighbour zero-range processes with non-homogeneous jump rates.
The class of  ``environments'' considered is close to that considered by \cite{afgl}, 
while our class of processes is broader. We also give  in arbitrary dimension  
a  simpler proof of the result of \cite{fs} with weaker assumptions. 
\end{abstract}
\section{Introduction}
\label{sec_intro}
Since \cite{jl, ev}, the study of disorder-induced phase transitions in driven 
lattice gases has attracted sustained interest, both in the mathematics and 
physics literature.  The model studied there was the totally asymmetric simple
 exclusion process (TASEP), respectively with a single defect site, and with i.i.d. particle disorder, 
the latter being equivalent 
to the totally asymmetric zero-range process (TAZRP) with constant jump rate and i.i.d. site disorder. 
Among subsequent works, TASEP with i.i.d. site disorder was considered 
in \cite{tb}, and  TAZRP with a single defect site
in \cite{lan2}.
More recently, dynamics-induced phase transition was studied (\cite{lan3}) 
in finite non-attractive homogeneous TAZRP, as well as the interplay between 
disorder-induced and dynamics-induced condensation (\cite{gl}).\\ \\
The aforementioned models have a single conserved quantity. In such cases, 
the usual picture is that  there exists a unique extremal invariant measure with 
given asymptotic density. This can be established rigorously for translation-invariant 
attractive processes with product invariant measures (\cite{and, coc, lig}). 
 When invariant measures are not explicit (\cite{EG, bgrs2, rez2}), one can only 
 show uniqueness, but it is not known whether such a measure exists for all density values. 
Among natural questions is the domain of attraction of each invariant measure.  
We are concerned here with the domain of attraction of the critical invariant  
measure for a class of  asymmetric zero-range processes (AZRP)  with quenched 
disorder when phase transition occurs.
For asymmetric models, large-time convergence is less well understood than for the 
symmetric exclusion process, where complete results are available (\cite{ligbook}). 
Even for translation-invariant systems, most of the precise results are one-dimensional 
and assume a translation-invariant initial distribution. Few exceptions considering 
deterministic initial configurations are \cite{brm, bl} for convergence to blocking 
measures, and \cite{BM} for convergence to translation-invariant measures.
\\ \\
 For models with mass as the only conserved quantity, phase transition is 
 defined  (\cite{cg})  as the existence of an interval of densities for which 
 an extremal invariant measure does not exist. 
For the site-disordered  AZRP 
this transition occurs when slow sites are rare enough, and for a finite system above 
critical density, the steady state is obtained  by completing the critical steady state 
with a Bose-Einstein condensate at the slowest site (\cite{ev}).  
For the infinite system  (with a drift, say, to the right),  the expected 
picture is the following (\cite{kru, jb}). (1) There is no invariant measure above 
some finite critical density $\rho_c$.   (2) Starting from an initial configuration 
with supercritical density to the left, 
(a) growing condensates of mesoscopic size appear at sites slower than their local 
environment on the left. Each condensate disappears as soon as it enters the domain 
of influence of a slower condensate on its left. Intervals between condensates are 
at critical equilibrium. (b) Eventually, all the supercritical mass will escape at 
$-\infty$, while the distribution of the microscopic state near the origin converges 
to the critical invariant measure.\\ \\
 On the mathematical side, some results related to (2a) can be found in \cite{gks}. 
A weaker form of (2b) is 
established in \cite{fs} for nonzero mean AZRP with i.i.d. site-disorder in any dimension. 
It is proved there that asymptotically, the local distribution of the process near the origin 
can never exceed the critical invariant measure, but no convergence is obtained. Besides, 
a subexponential growth condition at infinity is required on the initial configuration.
Statements (1) and (2b) are proved in \cite{afgl} for the totally asymmetric 
zero-range process with constant jump rate.  However, the approach used there 
does not extend to zero-range processes with more general jump kernels or more general 
jump rate functions, and the convergence to critical measure is established only for 
{\em strictly} supercritical initial conditions. \\ \\
In the present paper, for nearest-neighbour AZRP with site-disorder, within a large 
class of jump rate functions,
we establish convergence to the critical invariant measure when starting from an 
initial configuration whose asymptotic density to the left of the origin is at 
least critical.  Our improvement with respect to \cite{afgl} is thus 
threefold: first, we allow jumps to the right {\em and} left; next, 
the jump rate is no longer restricted to be constant; lastly, we also cover the case of
{\em critical} initial conditions. 
 Our result is  optimal in two respects. First, for a given process, 
 it provides a necessary and sufficient condition for convergence to the critical 
 measure (see \cite{bmrs1} for the proof of  necessity).  To our knowledge, this 
 is the first time a domain of attraction is completely identified for a conservative 
 system with nonzero drift. 
Next, our nearest-neighbour assumption on the kernel 
is the best possible: indeed, we showed in \cite{bmrs1} that the result may not 
hold for non nearest-neighbour kernels.
Like \cite{afgl}, we are not restricted to i.i.d. disorder, but work under sufficient 
conditions on a given environment. Our conditions on the disorder are pointwise slightly
 more restrictive than those of \cite{afgl}, but they equally include the case of a 
 random ergodic environment.\\ \\
To obtain our result, we prove an upper and a lower bound for limiting distributions as $t\to+\infty$. 
Our upper bound, valid in any space dimension, is of independent interest. Indeed, with a 
surprisingly short proof, we improve the result of \cite{fs} in two respects. First, 
we introduce a natural condition on a given environment and jump kernel; the latter is 
no longer required to have non-zero drift. Next, we remove the growth condition at infinity 
on the initial configuration.\\ \\
 Our approach of the lower bound is  based on  hydrodynamic limits. In our setting, 
 we cannot rely  on existing results. Indeed, the hydrodynamic limit of disordered AZRP 
 in the supercritical regime has been established so far in the case of TAZRP with constant 
 jump rate (\cite{ks}), but is still an open problem for more general AZRP. Let us recall 
 that for asymmetric models, the hydrodynamic limit is usually given (\cite{rez}) by 
 entropy solutions of a hyperbolic conservation law of the form
\be\label{burgers_intro}
\partial_t \rho(t,x)+\partial_x [f(\rho(t,x))]=0
\ee
where $\rho(t,x)$ denotes the local particle density. The function 
$\rho\mapsto f(\rho)$ in \eqref{burgers_intro} is the {\em flux function}, 
defined as the mean current in a system starting from a configuration with 
asymptotic mean density $\rho$.
For attractive models, phase transition implies (\cite{EG, rez2, bgrs2}) 
that $f$ is linear on any interval where invariant measures are missing.
In the present case, it is constant above critical density.
For general disordered AZRP, the hydrodynamic limit \eqref{burgers_intro}
was previously established in \cite{bfl} in any space dimension, but only 
below critical density. Here we need and establish  (see 
Proposition \ref{th_strong_loc_eq} below)  the hydrodynamic limit of a source, 
which is typically a supercritical process. Another question which has not been 
addressed yet for disordered AZRP is (strong) local equilibrium. A general approach 
was set up in \cite{lan} for translation-invariant models, but cannot be adapted 
outside this setting. We also obtain  
in Proposition \ref{th_strong_loc_eq}   a quenched strong local equilibrium 
result for the source process. The ideas used here are first steps towards
a more general proof of hydrodyamic limit and local equilibrium in disordered AZRP. 
We plan to address these issues in a forthcoming paper.
Let us mention, among related known results,
the hydrodynamic limit of TASEP with i.i.d. site disorder (\cite{sep}),
the occurrence of a plateau for the corresponding flux (\cite{bb}),
and the hydrodynamic limit of more general exclusion-like  attractive models 
with ergodic disorder (\cite{bgrs4}). \\ \\
The paper is organized as follows. In Section \ref{sec_results}, we introduce the notation 
and state our main results. 
Section \ref{sec_afgl} is devoted to the proof of the upper bound.
Section \ref{sec_proof} establishes the lower bound, and in 
Section \ref{sec_loc_eq} we prove the hydrodynamic 
limit results used in Section \ref{sec_proof}. 
\section{Notation and results}
\label{sec_results}
In the sequel,  $\R$ denotes the set of real numbers, 
$\Z$  the set of signed integers and $\N=\{0,1,\ldots\}$ 
the set of nonnegative integers. For $x\in\R$,  $\lfloor x\rfloor$  
denotes the integer part of $x$, that is the largest integer $n\in\Z$ 
such that $n\leq x$. The notation $X\sim\mu$ means that a random 
variable $X$ has probability distribution $\mu$.\\ \\
 Fix some $c\in(0,1)$. An environment (or disorder) is a 
 $(c,1]$-valued sequence  $\alpha=(\alpha(x),\,x\in\Z^d)$, for $d\ge 1$. 
 The set of environments
is denoted by ${\bf A}:=(c,1]^{\Z^d}$. 
Let $g:\N\to[0,+\infty)$ be a nondecreasing function such that
\be\label{properties_g}
g(0)=0<g(1)\leq \lim_{n\to+\infty}g(n)=:g_\infty<+\infty
\ee
We extend $g$ to $\overline{\N}:=\N\cup\{+\infty\}$  by setting $g(+\infty)=g_\infty$.
Without loss of generality, we henceforth assume $g_\infty=1$.
For $d\geq 1$, let  $\mathbf{X}:=\overline{\N}^{\Z^d}$  denote the set of particle configurations.
 This set, equipped with the product topology of $\overline{\N}$, is compact and metrizable. 
 A configuration is of the form  $\eta=(\eta(x):\,x\in\Z^d)$ 
 where $\eta(x)\in\overline{\N}$ for each $x\in\Z^d$. 
 The set $\mathbf{X}$ is equipped with the partial product order:  for $\eta,\xi\in\mathbf{X}$, we write 
$\eta\leq\xi$ if and only if $\eta(x)\leq\xi(x)$ for every $x\in\Z^d$.
Let $p(.)$ be a probability measure on $\Z^d$. 
Let $\alpha(.)$ be  a given realization of the disorder.
It follows from \cite[Theorem 3.9]{ligbook} (see  Appendix \ref{app:gen}  for details) that for any initial configuration $\eta_0\in\bf X$, there exists a unique in law Feller process $(\eta_t^\alpha)_{t\geq 0}$ on $\bf X$, such that $\eta_0^\alpha=\eta_0$, and with infinitesimal generator given
for any continuous  cylinder function $f:{\bf X}\to\R$ by
\be\label{generator}
L^\alpha f(\eta)  =  \sum_{x,y\in\Z^d}\alpha(x)
p(y-x)g(\eta(x))\left[
f\left(\eta^{x,y}\right)-f(\eta)
\right]\ee
where, if $\eta(x)>0$, $\eta^{x,y}:=\eta-\delta_x+\delta_y$ 
denotes the new configuration obtained from $\eta$ after a particle has 
jumped from $x$ to $y$ (configuration $\delta_x$ has one particle at $x$ and 
no particle elsewhere; addition of configurations is meant coordinatewise). 
 In cases of infinite particle number,  the following interpretations hold:
  $\eta^{x,y}=\eta-\delta_x$ if $\eta(x)<\eta(y)=+\infty$, 
  $\eta^{x,y}=\eta+\delta_y$ if $\eta(x)=+\infty>\eta(y)$,
$\eta^{x,y}=\eta$ if $\eta(x)=\eta(y)=+\infty$.  
Besides, if one starts with a configuration $\eta_0\in\N^{\Z^d}$, then 
with probability one, $\eta_t\in\N^{\Z^d}$ for every $t>0$. 
See Appendix  \ref{app:gen}  for more details on these statements. 
Another construction of this process from a space-time Poisson measure 
will be described in Subsection \ref{subsec:Harris} and shortly explained 
in Appendix \ref{app:graphical}, where the equivalence of the two definitions
 is also argued. The interested reader is referred to \cite{sw} for a unified 
 treatment of both constructions and their equivalence.  The framework of \cite{sw} 
 covers the present model thanks to the assumption that $g$ is bounded.  \\ \\
In the sequel, $\Prob$ will denote the underlying probability measure on a probability 
space $(\Omega,\mathcal F,\Prob)$ on which
the above process is constructed, and $\Exp$ will denote expectation with respect to $\Prob$. 
At this stage these spaces are unspecified,
and a more precise choice has no impact on the forthcoming statements of this section.
The reader may consider for instance the canonical Skorokhod space 
$\Omega=\mathcal D([0,+\infty),{\bf X})$, in which case $\Prob$ is simply the law of 
the process. However, in Subsection \ref{subsec:Harris} and subsequently, we shall be 
working with a specific choice of $\Omega$.\\ \\
For $\lambda<1$, we define the probability measure $\theta_\lambda$ on $\N$
 by \label{properties_a}
\begin{equation}\label{eq:theta-lambda}
\theta_\lambda(n):=Z(\lambda)^{-1}\frac{\lambda^n}{g(n)!},\quad n\in\N
\end{equation}
 where $g(n)!=\prod_{k=1}^n g(k)$ for $n\geq 1$, $g(0)!=1$, and $Z(\lambda)$ is the normalizing factor:
\be\label{def_Z}
Z(\lambda):=\sum_{n=0}^{+\infty}\frac{\lambda^n}{g(n)!}
\ee
We extend $\theta_\lambda$ into a probability measure on $\overline{\N}$ 
by setting $\theta_\lambda(\{+\infty\})=0$.
For $\lambda\leq c$, we denote by $\mu_\lambda^\alpha$ the invariant measure 
of $L^\alpha$ defined (see e.g. \cite{bfl}) as the product measure on $\mathbf{X}$ 
with one-site marginal $\theta_{\lambda/\alpha(x)}$. Since 
$(\theta_\lambda)_{\lambda\in [0,1)}$ is an exponential family, we have that
 \be\label{stoch_inc}
 \mu^\alpha_\lambda\,\mbox{ is weakly continuous and stochastically increasing with respect to }\lambda,
 \ee
and that the mean value of $\theta_\lambda$, given by 
\be\label{mean_density}
R(\lambda):=\sum_{n=0}^{+\infty}n\theta_\lambda(n)
\ee
is a $C^\infty$ increasing function from $[0,1)$ to $[0,+\infty)$.
The quenched mean particle density at $x$ under $\mu_\lambda^\alpha$ is defined by
\be\label{mean_density_quenched}
R^\alpha(x,\lambda):=\int\eta(x)d\mu_\lambda^\alpha(\eta)
=R\left(\frac{\lambda}{\alpha(x)}\right)
\ee
 In the forthcoming statements, $\eta_0\in\N^{\Z^d}$  denotes the initial particle configuration,
and $(\eta_t^\alpha)_{t\geq 0}$ the evolved quenched process with generator 
\eqref{generator} starting from $\eta_0$ in the environment $\alpha\in{\bf A}$.\\ \\
We can now state our results. First, we establish  a general upper bound which  
improves the result of \cite{fs}.
It is established there  for i.i.d. environments and 
jump kernel $p(.)$ with nonzero drift,  under an additional assumption on the initial configuration:
\be
\label{cond_fs_bis}
\sum_{n\in\N}e^{-\beta n}\sum_{x:\,|x|=n}\eta_0(x)<+\infty,\quad\forall\beta>0
\ee
Here we provide a   shorter  proof of the same result without  assumption
\eqref{cond_fs_bis}, and with an explicit assumption for a fixed environnment, 
which also includes cases with zero drift. 
\begin{theorem}\label{proposition_afgl}
Assume that  a given environment $\alpha$ is such that  
for every $x\in\Z^d$, the limit
\be\label{hyp_upper}
\liminf_{n\to+\infty}\alpha(\check{X}_n^x)=c
\ee
holds a.s., where  $(\check{X}_n^x)_{n\in\N}$ denotes the random walk on 
$\Z^d$ with kernel $\check{p}(.)=p(-.)$
starting from $x$. 
Then, for  
 every $\eta_0\in\N^{\Z^d}$  and every bounded local nondecreasing function $h:\mathbf{X}\to\R$,  
\be\label{upper_bound}
\limsup_{t\to\infty}\Exp h(\eta_t^\alpha)\leq\int_\mathbf{X} h(\eta)d\mu_c^\alpha(\eta)
\ee
\end{theorem}
 In the case of finite nonzero drift, 
if we think of the random walk paths on  large scale as straight lines, 
then \eqref{hyp_upper} means that there are enough slow sites (i.e. with rates 
close to $c$) at infinity opposite the drift direction so that one has a 
$(d-1)$-dimensional barrier of slow sites at infinity. This barrier acts as a source
that carries the critical density and hence bounds the possible densities of the system.  
Notice that if $d=1$  and the drift is nonzero,   a sufficient condition for \eqref{hyp_upper} is 
\be
\label{not_too_sparse}
\liminf_{x\to-\infty}\alpha(x)=c
\ee
 if the drift is to the right, 
or $\liminf_{x\to+\infty}\alpha(x)=c$ if the drift is to the left.
The i.i.d. case studied in \cite{fs} is contained in Assumption 
\eqref{hyp_upper}, as shown by the  following corollary. 
\begin{corollary}\label{cor_upper}Assume the random variables $\alpha(x)$ are 
 $(c,1]$ valued and i.i.d., where $c\in(0,1)$ is  the infimum of 
 the support  of their common distribution. 
 Then \eqref{hyp_upper} is satisfied by a.e. realization of these variables.
\end{corollary}
{}From now on, we let $d=1$, and  consider 
a nearest-neighbour jump kernel with non-zero drift, that is 
\be\label{nn_hyp}
p(1)=p\in(1/2,1],\quad p(-1)=q:=1-p
\ee
This is not a technical artefact: it is indeed shown in \cite{bmrs1} that 
the forthcoming results are wrong for more general kernels.\\ \\
For our main theorem  (Theorem \ref{convergence_maximal} below), 
we now introduce a set of conditions to be satisfied by the environment  $\alpha$. 
 First, the set of slow sites should not be too sparse. To this end we require 
\be\label{assumption_afgl}
\forall\varepsilon\in(0,1),\quad 
\lim_{n\to+\infty}\min\{\alpha(x):\,x\in\Z\cap[-n,-n(1-\varepsilon)]\}=c
\ee
This is equivalent to the existence of a sequence $(x_n)_{n\in\N}$ of sites such that
\be\label{assumption_afgl_2}
x_{n+1}<x_n<0,\quad \lim_{n\to +\infty}\frac{x_{n+1}}{x_n}=1,\quad \lim_{n\to+\infty}\alpha(x_n)=c
\ee
Assumption \eqref{assumption_afgl} implies \eqref{not_too_sparse}.  
Next, we assume existence of an annealed mean density:
\begin{eqnarray}\label{average_afgl}
\overline{R}(\lambda) & := & \lim_{n\to+\infty}
\frac{1}{n+1}\sum_{x=-n}^0 R\left(\frac{\lambda}{\alpha(x)}\right)\\
& = & \lim_{n\to+\infty}\frac{1}{n+1}\sum_{x=0}^n R\left(\frac{\lambda}{\alpha(x)}\right)
\quad\mbox{exists for every }\lambda\in[0,c)\nonumber
\end{eqnarray}
Note that the positive side  (w.r.t. the origin)  of Assumption 
\eqref{average_afgl} is necessary, while \eqref{assumption_afgl} or
 \eqref{assumption_afgl_2}  concerns only the negative side. Roughly speaking, 
 for our approach, we only need to find slow sites at $-\infty$, while we have 
 to prove a hydrodynamic limit statement also (slightly) to the right of $0$ 
 (see Remark \ref{remark_bilateral}).
 It can be shown  (see Lemma \ref{lemma_properties_flux} below and \cite{bmrs1})  that
 $\overline{R}$ is an increasing $C^\infty$ function on $[0,c)$.
We define the critical density by 
\be\label{other_def_critical}
\rho_c:=\overline{R}(c-):=\lim_{\lambda\uparrow c}\overline{R}(\lambda)\in[0,+\infty]
\ee
In the sequel, we extend $\overline{R}$ by continuity to $[0,c]$ by defining 
$$
\overline{R}(c):=\overline{R}(c-)=\rho_c
$$
Note that the value obtained by plugging $\lambda=c$ into \eqref{average_afgl}
 would not be a relevant definition of $\overline{R}(c)$, see \cite{bmrs1} for details.
Our next assumption is finiteness of the critical density:
\be\label{assumption_pt}
\rho_c=\overline{R}(c-)=\lim_{\lambda\uparrow c}\overline{R}(\lambda)
=\overline{R}(c)<+\infty
\ee
Finally, we need the following convexity assumption:\\ \\
(H) For every $\lambda\in[0,c)$, 
$\overline{R}(\lambda)-\overline{R}(c)-(\lambda-c)\overline{R}^{'+}(c)>0$\\ \\
where
\be\label{def_upperdiff}
\overline{R}^{'+}(c):=\limsup_{\lambda\to c}\frac{\overline{R}(c)-\overline{R}(\lambda)}{c-\lambda}
\ee
is the left-hand derivative at $c$ of the convex envelope of $\overline{R}$
(notice that our assumptions do not imply existence of the derivative $\overline{R}'(c)$).\\ \\
For instance, if $R$ is strictly convex, then for any environment 
satisfying \eqref{assumption_afgl}--\eqref{average_afgl}, 
$\overline{R}$ is strictly convex  (see Lemma \ref{lemma_properties_flux} below), 
and thus (H) satisfied. A sufficient condition for $R$ to be strictly convex 
(see \cite[Proposition 3.1]{bs}) is that 
\be\label{convex_micro}
n\mapsto g(n+1)-g(n)\mbox{ is a nonincreasing function}
\ee
Assumption \eqref{assumption_afgl_2}, due to its second condition, 
is slightly stronger than assumption \eqref{not_too_sparse} made in \cite{afgl}. 
Note that both are equivalent if we assume that $\alpha$ is a typical realization 
of an ergodic random environment. However, \eqref{assumption_afgl_2} includes 
environments with zero density of defects, which cannot be obtained as a realization 
of an ergodic disorder with infimum value $c$. A typical example is an environment 
$\alpha(.)$ for which a sequence satisfying \eqref{assumption_afgl_2} exists,  
with  $\lim_{n\to+\infty}n^{-1}x_n=-\infty$,  and $\alpha(x)=1$ for $x\in\Z$ not belonging 
to this sequence. In this case, 
one has $\overline{R}(\lambda)=R(\lambda)$ and $\rho_c=R(c)<+\infty$. 
The second condition in \eqref{assumption_afgl_2} sets a restriction on the admitted
sparsity of defects, in the sense that their empirical density must decay less 
than exponentially in space. We believe that this condition is an artefact 
of our approach, but this is the small price we can currently not avoid 
for extending the result of \cite{afgl} to partial asymmetry and more general functions $g$. \\ \\
We also expect condition (H)  to be  an artefact of our method, and the convergence Theorem
\ref{convergence_maximal} below to hold without it.  A similar condition appears also 
in the proof of \cite{afgl}  (on line 8, page 77),  where it is automatically implied 
by the more stringent restriction $g(n)=\indicator{\{n>0\}}$. \\ \\
 Having introduced the above conditions, we are ready to state the main result of this paper.
\begin{theorem}
\label{convergence_maximal}
Assume  \eqref{nn_hyp},  and that $\alpha$ is a given environment satisfying  
\eqref{assumption_afgl}, \eqref{average_afgl}, \eqref{assumption_pt} and (H). Then, 
for any  $\eta_0\in\N^\Z$  satisfying the supercriticality assumption
\be
\label{cond_init}
\liminf_{n\to\infty}n^{-1}\sum_{x=-n}^0\eta_0(x)\geq \rho_c,
\ee
the quenched process $(\eta_t^\alpha)_{t\geq 0}$ with initial 
state $\eta_0$ converges in distribution to $\mu_c^\alpha$ as $t\to\infty$.
\end{theorem}
 The reader may wonder if there is a simple example of a {\em given} 
 environment $\alpha$ satisfying 
\eqref{average_afgl} and \eqref{assumption_afgl} (or \eqref{assumption_afgl_2}). 
Here is one. Let $Q$ be a probability measure on $(c,1]$.
Divide the set of negative integers into  a disjoint union of intervals 
$I_j=\Z\cap[x_{j+1},x_j-1]$, where $j\in\N$, and
$(x_j)_{j\in\N}$ is a decreasing sequence  such that 
\be\label{length_infinity}\lim_{j\to+\infty}(x_j-x_{j+1})=+\infty\ee
For $k=1,\ldots,x_j-x_{j+1}=|I_j|$, let $x_{j,k}:=x_j-k$. We now define $\alpha(.)$ 
as follows. First, define  $\beta(x)$ for $x\neq 0$ by specifying that, for every 
$j\in\N$ and $k=1,\ldots,x_j-x_{j+1}$, 
$$
\beta(x_{j,k}) =
\beta(-x_{j,k}) := \frac{k}{x_j-x_{j+1}+1}
$$
Then, for every $x\neq 0$, set 
$$
\alpha(x) :=  F_{Q}^{-1}[\beta(x)]
$$
where $F_Q$ is the distribution function of $Q$, defined by 
$F(t):=Q(-\infty,t]$ for $t\in\R$, and $F_Q^{-1}$ its left-continuous inverse defined by
$$
F_Q^{-1}(u):=\inf\{t\in\R:\,F(t)\geq u\}=\sup\{t\in\R:\,F(t)<u\}
$$
To complete this definition, give an arbitrary value in $(c,1]$ to $\alpha(0)$.
Let $t\in[0,1]$. By construction,  given \eqref{length_infinity},  
the number of values $\beta(x)$ not exceeding $t$ for $x\in I_j$ 
is of order $t|I_j|$ as $j\to+\infty$, and
the total length of intervals $I_1,\ldots,I_n$ is $x_0-x_n$. 
Thus
$$
\lim_{n\to+\infty}\frac{1}{|x_n|}\sum_{x=x_n}^{1}{\bf 1}_{\{\beta(x)\leq t\}}=
\lim_{n\to+\infty}\frac{1}{|x_n|}\sum_{x=1}^{-x_n}{\bf 1}_{\{\beta(x)\leq t\}}
=t
$$
We now assume in addition that
\be\label{assume_in_addition}
\lim_{j\to+\infty}\frac{x_{j+1}}{x_j}=1
\ee
Then we have more generally 
$$
\lim_{n\to+\infty}\frac{1}{n}\sum_{x=-n}^{1}{\bf 1}_{\{\beta(x)\leq t\}}=
\lim_{n\to+\infty}\frac{1}{n}\sum_{x=1}^{n}{\bf 1}_{\{\beta(x)\leq t\}}
=t
$$
and thus, since for all $t\in\R$, $\alpha(x)\leq t\Leftrightarrow \beta(x)\leq F_Q(t)$,
$$
\lim_{n\to+\infty}\frac{1}{n}\sum_{x=-n}^{1}{\bf 1}_{\{\alpha(x)\leq t\}}=
\lim_{n\to+\infty}\frac{1}{n}\sum_{x=1}^{n}{\bf 1}_{\{\alpha(x)\leq t\}}
=F_Q(t)
$$
It follows that
\be\label{limit_empirical}
\lim_{n\to+\infty}\frac{1}{n}\sum_{x=1}^n\delta_{\alpha(x)}
=\lim_{n\to+\infty}\frac{1}{n}\sum_{x=-n}^{-1}\delta_{\alpha(x)}=Q
\ee
in the sense of weak convergence of measures. 
Thus \eqref{average_afgl} is satisfied with
\be\label{RQ}
\overline{R}(\lambda):=\int_{(c,1]}R\left(
\frac{\lambda}{\alpha}
\right)Q(d\alpha)
\ee
for every $\lambda\in[0,c)$, since for every such $\lambda$, 
the above integrand is a bounded function of $\alpha\in(c,1]$. Besides,  condition 
\eqref{assumption_afgl_2} is satisfied by our sequence $(x_n)_{n\in\N}$ 
since we assumed \eqref{assume_in_addition}, and the third condition in 
\eqref{assumption_afgl_2} follows from the fact that $\alpha(x_n)=F_Q^{-1}(1/n)$ 
converges to $c$ as $n\to+\infty$.
Note that the function $\overline{R}$ in \eqref{RQ} is also the one we would 
obtain if $\alpha$ were a typical realization of i.i.d. (or more generally 
ergodic) disorder with marginal distribution $Q$ for $\alpha(0)$.\\ \\
The following variant of the above construction shows an example of a {\em random} 
independent but not ergodic environment for which 
\eqref{average_afgl} and \eqref{assumption_afgl} (or \eqref{assumption_afgl_2}) 
hold almost surely. Given a sequence $(x_j)_{j\in\N}$ and intervals $I_j$ as above, 
and a sequence of i.i.d. random variables $(a_i)_{i\in\Z}$ with marginal distribution $Q$, 
we set $\alpha(x)=a_{-j}$ for $x\in I_j$  and $\alpha(x)=a_{j}$ for $x\in -I_j$, 
and define $\alpha(0)$ to be a $(c,1]$-valued random variable independent of $(a_j)_{j\in\Z}$.\\ \\
 Given  the upper bound of Theorem \ref{proposition_afgl}, the remaining  
part of the work will be to prove the following lower bound,  
for which all assumptions of Theorem \ref{convergence_maximal} are in force. 
This will be done   in Sections \ref{sec_proof} and \ref{sec_loc_eq}. 
\begin{proposition}\label{prop_main_work}
Assume  \eqref{nn_hyp},   and that $\alpha$ is a given environment 
satisfying  \eqref{assumption_afgl}, \eqref{average_afgl}, \eqref{assumption_pt}  and  (H). 
Then the following holds: 
for any $\eta_0\in{\N^\Z}$  satisfying \eqref{cond_init}, 
and every bounded local nondecreasing function $h:\mathbf{X}\to\R$,
\be\label{lower_bound}
\liminf_{t\to\infty}\Exp h(\eta_t^\alpha)\geq\int_\mathbf{X} h(\eta)d\mu_c^\alpha(\eta)
\ee
\end{proposition}
\section{Proof of Theorem \ref{proposition_afgl} and Corollary \ref{cor_upper}}\label{sec_afgl}
 We start recalling standard material in Subsections \ref{subsec_harris} 
 and \ref{subsec_jackson}, before proceeding to the actual proofs
in Subsection \ref{subsec_actual}.
\subsection{Harris construction and coupling}\label{subsec_harris}\label{subsec:Harris}
Let us recall the Harris construction of the process with infinitesimal generator \eqref{generator}.
We introduce a probability space $(\Omega,\mathcal F,\Prob)$, whose
generic element $\omega$ - called a Harris system (\cite{har}) - of $\Omega$  
is a  locally finite point measure of the form
\be\label{def_omega}
\omega(dt,dx,du,dz)=\sum_{n\in\N}\delta_{(T_n,X_n,U_n,Z_n)}
\ee
 on  $(0,+\infty)\times\Z^d\times(0,1)\times\Z^d$,  
 where $\delta$ denotes Dirac measure,  and $(T_n,X_n,U_n,Z_n)_{n\in\N}$ is a 
$(0,+\infty)\times\Z^d\times(0,1)\times\Z^d$-valued sequence.  
Under the probability measure $\Prob$, $\omega$ is a Poisson measure with 
intensity 
\be\label{def_intensity}\mu(dt,dx,du,dz):=dtdx\indicator{[0,1]}(u)du\, p(z)dz\ee
 In the sequel, the notation $(t,x,u,z)\in\omega$ will mean $\omega(\{(t,x,u,z)\})=1$. 
 We shall also say that $(t,x,u,z)$ is a potential jump event. 
An alternative interpretation of this random point measure is that 
we have three mutually independent families of independent 
random variables $(D^x_k)_{x\in\Z,\,k\in\N}$, $(U^x_k)_{x\in\Z,\,k\in\N}$ and
$(Z^x_k)_{x\in\Z,\,k\in\N}$,  such that $D^x_k$ has exponential distribution 
with parameter $1$, $U_k^x$ has uniform distribution on $(0,1)$, 
$Z^x_k$  has distribution  $p(.)$, and that if we set
\be\label{poisson_one}
T^x_k:=\sum_{j=0}^k D^x_j,
\ee
then, $\Prob$-a.s.,
\be\label{notation_harris}
\omega(dt,dx,du,dz)=\sum_{x\in\Z}\sum_{k\in\N}\delta_{(T_k^x,x,U_k^x,Z^x_k)}
\ee
On  $(\Omega,\mathcal F,\Prob)$,  a c\`adl\`ag process $(\eta_t^\alpha)_{t\geq 0}$ 
with generator \eqref{generator} and initial configuration $\eta_0$ can be constructed
in a unique way so that 
\be\label{rule_1}
\forall(s,x,z,v)\in\omega,\quad
v\leq \alpha(x)g\left[\eta^\alpha_{s-}(x)\right]  \Rightarrow \eta^\alpha_s=(\eta^\alpha_{s-})^{x,x+z}
\ee
and, for all $x\in\Z^d$ and $0\leq s\leq s'$,
\be\label{rule_2}
\omega\left(
(s,s']\times E_x
\right)=0\Rightarrow\forall t\in(s,s'],\,\eta_t(x)=\eta_s(x)
\ee
where
\be\label{def_E_x}
E_x:=\{(y,u,z)\in\Z^d\times(0,1)\times\Z^d:\,
x\in\{y,y+z\}
\}
\ee
(note that the inequality in \eqref{rule_1} implies $\eta_{t-}^\alpha(x)>0$, 
cf. \eqref{properties_g}, thus $(\eta^\alpha_{t-})^{x,x+z}$ is well-defined). 
Equation \eqref{rule_1} says when a potential jump event gives rise to an actual jump, 
while \eqref{rule_2} states that no jump ever occurs outside potential jump events.
For reader's convenience, in Appendix \ref{app:graphical}, we briefly sketch 
this construction and the reason why it yields the same process as \eqref{generator}. \\ \\
When  necessary, random initial conditions are constructed on an auxiliary probability space $\Omega_0$ 
equipped with a probability measure $\Prob_0$. \\ \\
Expectation with respect to $\Prob$ (resp. $\Prob_0$) is denoted by $\Exp$ (resp. $\Exp_0$). 
The product space  $\Omega_0\times\Omega$  is equipped with the product measure and 
$\sigma$-fields (thus  environment, initial particle configuration and Harris system 
are mutually  independent). Joint expectation with respect to the product 
measure  is denoted by  $\Exp_0\Exp$. \\ \\
In the sequel, we shall have to couple different processes with different 
(possibly random) initial configurations and possibly different environments. 
Such couplings will be realized on $\Omega_0\times\Omega$ by using the same 
Poisson clocks for all processes.
The following result is a consequence of the monotonicity assumption on $g$.
\begin{proposition}\label{monotone_harris}
Let $\eta_0^1\in\mathbf{X}$ and $\eta_0^2\in\mathbf{X}$ be two initial configurations,  
$\alpha^1\in{\bf A}$ and $\alpha^2\in{\bf A}$ two environments.
If $\eta_0^1\leq\eta_0^2$, $\alpha^1(x)\leq\alpha^2(x)$ for every $x\in\Z^d$ 
such that $\eta_0^2(x)=+\infty$, and $\alpha^1(x)=\alpha^2(x)$ for every $x\in\Z^d$ such that
$\eta_0^2(x)<+\infty$, then $\eta_t^1\leq\eta_t^2$ for every $t>0$.
\end{proposition}
The above proposition contains in particular the monotonicity of Harris 
coupling (which implies attractiveness) for  identical  environments. 
If a process starts with $+\infty$ particles at sites $x\in S$ for some 
$S\subset \Z^d$, $S$ can be viewed as a set of source/sinks, and
$\Z^d\setminus S$  as the set on which particles live. This is the object of the next subsection.
\subsection{Open Jackson networks}\label{subsec_jackson}
Let $S\subset\Z^d$, and $(\eta_t)_{t\geq 0}$ a process with generator \eqref{generator} 
such that $\eta_0(x)=+\infty$ for all $x\in S$. Then $\eta_t(x)=+\infty$ for all $x\in S$ 
and $t\geq 0$, and the process 
$(\eta^{\alpha,S}_t)_{t\geq 0}$, where $\eta_t^{\alpha,S}$  denotes the restriction 
of $\eta^\alpha_t$ to  $\Z^d\setminus S$,  is itself a Markov process on 
$\overline{\N}^{\Z^d\setminus S}$ with generator
\begin{eqnarray}\nonumber
L^{\alpha,S}f(\eta)&=&\sum_{x,y\not\in S}
\alpha(x)p(y-x)g(\eta(x))[f(\eta^{x,y})-f(\eta)]\\
\nonumber & + & \sum_{x\in S,\,y\not\in S} \alpha(x)p(y-x)g(\eta(x))[f(\eta+\delta_y)-f(\eta)]\\
\label{gen_open} & + & \sum_{x\not\in S,\,y\in S} \alpha(x)p(y-x)g(\eta(x))[f(\eta-\delta_x)-f(\eta)]
\end{eqnarray} 
We may identify the process on $\overline{\N}^{\Z^d}$ and the one on 
$\overline{\N}^{\Z^d\setminus S}$, since the restriction of the former 
to $\overline{\N}^S$ is  identically  $+\infty$. Proposition 
\ref{monotone_harris} can be rephrased as follows
 in this setting.
\begin{proposition}\label{monotone_jackson}
If $S\subset S'\subset\Z^d$,  $\alpha=\alpha'$ on $\Z^d\setminus S'$  and 
$\alpha\leq\alpha'$ on $S'$, a process $(\eta_t^{\alpha,S})_{t\geq 0}$ 
with generator $L^{\alpha, S}$ and a process  $(\eta_t^{\alpha',S'})_{t\geq 0}$ 
with generator $L^{\alpha', S'}$
can be coupled in such a way that
\be\label{couple_modified_source_0}
\eta^{\alpha,S}_0\leq\eta^{\alpha',S'}_0\mbox{ on }\Z^d\setminus S'\Longrightarrow\forall t>0,
\eta^{\alpha,S}_t\leq\eta^{\alpha',S'}_t\mbox{ on }\Z^d\setminus S'
\ee 
\end{proposition}
The  process defined by \eqref{gen_open} is an open Jackson network, 
whose invariant measures are well-known in queuing theory. 
\begin{proposition}\label{prop_jackson}
Consider the system
\begin{eqnarray}
\lambda(x) & = & \sum_{y\in \Z^d}\lambda(y)p(x-y),
\quad\mbox{for all }x\in\Z^d\setminus S\label{jackson_bulk}\\
\lambda(x) & = & \alpha(x),\quad\mbox{for all }x\in S\label{jackson_boundary}
\end{eqnarray}
Assume $\lambda(.)$ is a solution of \eqref{jackson_bulk}--\eqref{jackson_boundary} such that
\be\label{cond_rec}
\lambda(x)<\alpha(x),\quad\mbox{for all }x\in\Z^d\setminus S
\ee
For any $S$, $\alpha(.)$ and $\lambda(.)$ satisfying 
\eqref{jackson_bulk}--\eqref{jackson_boundary} and \eqref{cond_rec}, 
the product measure $\mu^{\alpha,S,\lambda}$  on $\N^{\Z^d\setminus S}$ 
with marginal $\theta_{\lambda(x)/\alpha(x)}$ at site $x\in\Z^d\setminus S$
is invariant for $L^{\alpha,S}$.\\ \\
If in addition $\Z^d\setminus S$ is finite,
the system \eqref{jackson_bulk}--\eqref{jackson_boundary} has a 
unique solution $\lambda^{\alpha,S}(.)$ given by
\be\label{solution_system}
\lambda^{\alpha,S}(x)=\check{\Exp}_x \alpha(\check{X}_T)\indicator{\{\check{T}<+\infty\}}
\ee
where $\check{T}$ denotes the hitting time of $S$ by the random walk 
$(\check{X}_n)_{n\in\N}$ with reversed kernel 
$\check{p}(x):=p(-x)$, and  $\check{\Exp}_x$   denotes expectation 
with respect to the law $\check{\Prob}_x$ of this random walk starting from $x$.
Besides,
the restriction to  $\N^{\Z^d\setminus S}$  of the process with generator 
\eqref{gen_open} is positive recurrent if and only if this solution 
satisfies  condition \eqref{cond_rec}, in which case 
$\mu^{\alpha,S}:=\mu^{\alpha,S,\lambda^{\alpha,S}}$ is
its unique invariant measure.
\end{proposition}
\begin{proof}{Proposition}{prop_jackson}
The  first  statement follows from standard results (see e.g. \cite{par}). 
For uniqueness we have to verify the assumption
$\Prob_x(T<+\infty)=1$
for every $x\in\Z^d\setminus S$, where
$$
T:=\inf\{n\in\N:\,{X}_n\in S\}\in\overline{\R}
$$
denotes the hitting time of $S$ by the random walk $(X_n)_{n\in\N}$ with kernel $p$, 
and  $\Prob_x$ denotes the law of this random walk starting from $x$.
This follows from the fact that a random walk on $\Z^d$ a.s. leaves any finite set 
in finite time, unless its jump kernel is supported on $\{0\}$.  But since 
$\alpha(x)>c$ for all $x\in\Z^d$, this case is incompatible with assumption 
\eqref{hyp_upper}.  Finally, \eqref{solution_system} is a known 
solution of \eqref{jackson_bulk}--\eqref{jackson_boundary}.
\end{proof}\\ \\
If the solution to \eqref{jackson_bulk}--\eqref{jackson_boundary} does not 
satisfy \eqref{cond_rec}, one cannot define the stationary measure. The following 
corollary shows how to modify the source so that it becomes possible.
\begin{corollary}
\label{cor_jackson}
  Assume $\lambda(.)$ is a solution of \eqref{jackson_bulk}--\eqref{jackson_boundary}. 
Define an augmented source set $S'=S'(\alpha,S,\lambda)$ by
\be\label{modif_S}
S':=S\cup \left\{
x\in\Z^d\setminus S:\,\lambda(x)\geq\alpha(x)
\right\}
\ee
and a modified environment
$\alpha'=\alpha'(\alpha,S,\lambda)$ by 
\be\label{modif_alpha}
\alpha'(x):=\left\{
\ba{lll}
\lambda(x) & \mbox{if} & x\in S'\setminus S\\
\alpha(x) & \mbox{otherwise} & 
\ea
\right.
\ee
Then $\lambda(.)$ satisfies \eqref{cond_rec} if $S$ and $\alpha$ are replaced by $S'$ and $\alpha'$, and
 $\mu^{\alpha',S'}:=\mu^{\alpha',S',\lambda}$ is an invariant measure for $L^{\alpha',S'}$. 
\end{corollary}
\begin{proof}{Corollary}{cor_jackson}
This results from the following observations. First, \eqref{cond_rec} 
is satisfied by definitions of $S'$ and $\lambda'$.
 Then $\Z^d\setminus S'$ is finite since $S\subset S'$. 
 Next, if $\lambda$, $\alpha$ and $S$ satisfy \eqref{jackson_bulk}--\eqref{jackson_boundary}, then 
$\lambda'$, $\alpha'$ and $S'$ still do. 
\end{proof}
\mbox{}\\ \\
The point of Corollary \ref{cor_jackson} is that $\alpha'\geq\alpha$ on $S'$.
 Thus, taking $S=\emptyset$ in \eqref{couple_modified_source_0}, we have a coupling 
 of $(\eta_t^\alpha)_{t\geq 0}$ and  $(\eta_t^{\alpha',S'})_{t\geq 0}$
such that
\be\label{couple_modified_source}
\eta^{\alpha}_0\leq\eta^{\alpha',S'}_0\mbox{ on }\Z^d\setminus S'\Longrightarrow\forall t>0,
\eta^{\alpha}_t\leq\eta^{\alpha',S'}_t\mbox{ on }\Z^d\setminus S' 
\ee 
We are now ready for the  proofs of Proposition \ref{proposition_afgl} and Corollary \ref{cor_upper}.
\subsection{Main proofs}\label{subsec_actual}
\begin{proof}{Theorem}{proposition_afgl}
Let $F$ be a finite subset of $\Z^d$ such that $h(\eta)$ depends on $\{\eta(x),\,x\in F\}$.
For $\varepsilon>0$  and $\delta>0$, let 
\be\label{def_source_set}
S_{\varepsilon,\delta}:=\{x\in\Z^d:\,\alpha(x)<c+\varepsilon\mbox{ or }|x|> \delta^{-1}\}
\ee
The complement of  $S_{\varepsilon,\delta}$  
 is finite because of the second condition. Hence, for any starting point 
 $x\in\Z^d\setminus S_{\varepsilon,\delta}$,  the hitting time  $\check{T}_{\varepsilon,\delta}$ 
of $S_{\varepsilon,\delta}$ by the random walk with kernel $\check{p}(.)$  is  a.s. finite. 
We write  $\check{T}_{\varepsilon,\delta}=\min(\check{U}_{\varepsilon},\check{V}_{\delta})$, where
\begin{eqnarray*}
\check{U}_{\varepsilon} & := & \inf\{n\in\N:\,\alpha(\check{X}_n)<c+\varepsilon\}\\
\check{V}_{\delta} & := & \inf\{n\in\N:\,\left|\check{X}_n\right|>\delta^{-1}\}
\end{eqnarray*}
 These hitting times are also $\Prob_x$-a.s. finite for any starting point, since
 $\check{V}_\delta$ is again the exit time from a finite set, while for
  $\check{U}_\varepsilon$ this follows from \eqref{hyp_upper}. 
For every $x\in F$, it holds that
$$
\lim_{\delta\to 0}\check{V}_{\delta}=+\infty,\quad\Prob_x\mbox{-a.s.}
$$
It follows that, 
for each $\varepsilon>0$, we can find $\delta=\delta(\varepsilon)>0$ such that, for every $x\in F$,
\be\label{exit_first_cond}
\lim_{\varepsilon\to 0}\check{\Prob}_x \left(
\check{T}_{\varepsilon,\delta}=\check{U}_{\varepsilon}
\right)=1
\ee
 In the sequel, in all expressions involving $\varepsilon$ and $\delta$, 
 we shall implicitely take $\delta=\delta(\varepsilon)$.
It follows from \eqref{solution_system} and \eqref{exit_first_cond} that 
 if $\lambda_{\varepsilon,\delta}:=\lambda^{\alpha,S_{\varepsilon,\delta}}$, 
\be\label{convergence_lambda}
\lim_{\varepsilon\to 0}
\lambda_{\varepsilon,\delta}(x)=c,\quad\mbox{for all }x\in F
\ee
 We may thus, and will henceforth, take $\varepsilon$ small enough to have
\be\label{F_OK}
\lambda_{\varepsilon,\delta}(x)<\alpha(x),\quad\mbox{for all }x\in F
\ee
Let
 $S'_{\varepsilon,\delta}:=S'(\alpha,S_{\varepsilon,\delta},\lambda_{\varepsilon,\delta})$ 
defined by \eqref{modif_S}, and 
$\alpha':=\alpha'(\alpha,S_{\varepsilon,\delta},\lambda_{\varepsilon,\delta})$  
defined by \eqref{modif_alpha}.
Note that \eqref{F_OK} implies 
\be\label{F_OK_2}
F\subset\Z^d\setminus S'_{\varepsilon,\delta}
\ee
By Corollary \ref{cor_jackson}, $\mu_{\varepsilon,\delta}:=\mu^{\alpha',S'_{\varepsilon,\delta}}$ 
 is an invariant measure for $L^{\alpha',S'_{\varepsilon,\delta}}$.  \\ \\
Let $(\eta^{\alpha',S'_{\varepsilon,\delta}}_t)_{t\geq 0}$ denote the process 
with generator $L^{\alpha',S'_{\varepsilon,\delta}}$, and whose initial configuration 
in $S'_{\varepsilon,\delta}$ is the restriction 
of $\eta_0$ to this set. By Proposition \ref{prop_jackson}, this process converges 
in distribution as $t\to+\infty$ to its invariant measure $\mu_{\varepsilon,\delta}$ defined above.
By \eqref{couple_modified_source}, $\eta_t^\alpha\leq\eta^{\alpha',S'_{\varepsilon,\delta}}_t$ 
on  $\Z^d\setminus S'_{\varepsilon,\delta}$. 
Because $h$ is nondecreasing,  recalling \eqref{F_OK_2}, we have 
$$
\limsup_{t\to \infty} \Exp h(\eta_t^\alpha) \leq
\lim_{t\to \infty} \Exp h(\eta_t^{\alpha',S'_{\varepsilon,\delta}})
=\int_\mathbf{X} h(\eta)d\mu_{\varepsilon,\delta}(\eta)
$$
By \eqref{convergence_lambda},
$$
\lim_{\varepsilon\to 0}\int_\mathbf{X} h(\eta)d\mu_{\varepsilon,\delta}(\eta)
=\int_\mathbf{X} h(\eta)d\mu^\alpha_c(\eta)
$$
This concludes the proof of Theorem \ref{proposition_afgl}.
\end{proof}\\ 
\begin{proof}{Corollary}{cor_upper}
By Fubini's theorem, it is enough to show that, for a.e. random walk path
 $(\check{X}_n)_{n\in\N}$, \eqref{hyp_upper} holds a.s. with respect to 
 the law of the environment. But for a.e. path realization, there exists
  an increasing subsequence $(n_k)_{k\in\N}$ such that
$\check{X}_{n_k}\neq \check{X}_{n_l}$ for $k\neq l$. Assume such a path fixed. 
Since the random variables $\{\alpha(\check{X}_{n_k});\,k\in\N\}$ are i.i.d.
and $c$ is the infimum of their support, \eqref{hyp_upper} holds with probability 
one with respect to the joint law of these variables. 
\end{proof}
\section{Proof of Proposition \ref{prop_main_work}}
\label{sec_proof}
 This proof, 
divided into several parts, is outlined in  
Subsection \ref{subsec_proof}. 
 Subsection \ref{subsec_loc_eq} states new hydrodynamic limit and strong local 
 equilibrium results for a source,  given in Proposition \ref{th_strong_loc_eq},
  which will be established in Section \ref{sec_loc_eq}. These results 
are the main ingredients in the proof of Proposition \ref{prop_main_work}.  
Subsections \ref{subsec_currents} (devoted to currents, for which 
various properties are needed) and \ref{subsec_super} contain proofs of intermediate results. 
\subsection{Outline of proof}\label{subsec_proof}
As a preliminary remark we observe that, by attractiveness, 
it is enough to prove \eqref{lower_bound} for $\eta_0$ 
satisfying \eqref{cond_init}  such that
\be\label{empty_right}\eta_0(x)=0\mbox{ for all }x>0\ee
This will be assumed from now on.\\ \\
Our aim is to derive a lower bound. However,
since $\eta_0$ can be very irregular, for example it could have large spikes  
and long stretches of empty sites, regular configurations (for example 
with subcritical density profiles) may not be useful to obtain bounds using attractiveness. 
Therefore our strategy to prove Proposition \ref{prop_main_work} is 
to compare $\eta_t^\alpha$ in the neighborhood of $0$ to the process $(\eta_s^{\alpha,t})_{s\geq 0}$
with initial configuration  (with the convention $(+\infty)\times 0=0$)
\be\label{def_init_source}
\eta^{\alpha,t}_0(x)  =  (+\infty)\indicator{\{x\leq x_t\}}
\ee
for an appropriate choice of $x_t$. This process is a semi-infinite process 
with a source/sink at $x_t$: with rate $p\alpha(x_t)$, a particle is created at $x_t+1$,
with rate $q \alpha(x_t+1)g(\eta(x_t+1))$ a particle at $x_t+1$ is destroyed.
While $\eta^{\alpha,t}_0$ bounds $\eta_0$ from above in the region 
to the left of $x_t$, we will show that near the origin  $\eta^\alpha_t$  dominates 
$\eta^{\alpha,t}_t$.
Furthermore we will establish that the distribution of $\eta^{\alpha,t}_t$ in
any finite domain around the origin is arbitrarily close to $\mu_c^\alpha$,
and in the limit as $t\to+\infty$ we will then obtain our result. To achieve this,
$x_t$ should be chosen so that in the hydrodynamic limit for the process 
$(\eta_s^{\alpha,t})_{s\geq 0}$, the macroscopic density at the origin dominates
any density lower than $\rho_c$.
\\ \\
We now define quantities relevant for the hydrodynamic limit of that process.
Let
\be\label{def_v0}
v_0:=(p-q)\inf_{\lambda\in [0,c)}\frac{c-\lambda}{\overline{R}(c)-\overline{R}(\lambda)}
\ee
As stated in Lemma \ref{lemma_entropy} below, 
$v_0$ can be interpreted as the speed of a front  of uniform density $\rho_c$ 
issued by  a  source. Assumption (H) is equivalent to the infimum in \eqref{def_v0} 
 being achieved uniquely for $\lambda$ tending to $c$, which in turn is equivalent to 
\be\label{equal_v0}
v_0=(p-q)\overline{R}^{'+}(c)^{-1}\in[0,+\infty)
\ee
where $\overline{R}^{'+}$ was defined in \eqref{def_upperdiff}.
\\ \\
Let  $\varepsilon>0$ and $\beta<-v_0$.
We have in mind that $\beta=\beta(\varepsilon)$ will be a function 
of $\varepsilon$ that tends to $-v_0$ as $\varepsilon\to 0$
(the choice of this function will appear below). 
For the main idea developed in this section, we let $x_t:=\lfloor \beta t\rfloor$, 
with more precision to come on $\beta$. However, for various purposes in the sequel 
of the paper, configurations of the type \eqref{def_init_source} may be used with 
a different choice of $x_t$.
We then proceed in two main steps as follows. We establish the comparison
between $\eta_t^\alpha$ and $(\eta_s^{\alpha,t})_{s\geq 0}$ in Lemma 
\ref{lemma_asymptotic_order}. Then in Lemma \ref{lemma_asymptotic_source}
we derive the result of Proposition \ref{prop_main_work} using the semi-infinite process.\\ \\
 For $\varepsilon>0$, let 
\begin{eqnarray}\label{def:ell}
A_\varepsilon:=A_\varepsilon(\alpha)&=&\max\{x\le 0:\alpha(x)\leq c+\varepsilon\}\\
a_\varepsilon:=a_\varepsilon(\alpha)&=&\min\{x\ge 0:\alpha(x)\leq c+\varepsilon\}\label{def:err}
\end{eqnarray}
It follows from definition \eqref{def:ell} that
\be\label{location}
\lim_{\varepsilon\to 0}A_\varepsilon(\alpha)=-\infty
\ee 
\begin{lemma}\label{lemma_asymptotic_order} 
 Assume  \eqref{nn_hyp},   and that $\alpha$ is a given environment satisfying  
 \eqref{assumption_afgl}, \eqref{average_afgl}, \eqref{assumption_pt}  and  (H).  
 There exists a function $\beta=\beta(\varepsilon)$ such that $\beta<-v_0$, 
 $\lim_{\varepsilon\to 0}\beta(\varepsilon)=-v_0$, and 
\be\label{asymptotic_order}
\lim_{\varepsilon\to 0}\liminf_{t\to+\infty}\Prob\left(
\left\{
\eta_t^\alpha(x)\geq\eta^{\alpha,t}_t(x),\,\forall x\geq A_\varepsilon(\alpha)
\right\}
\right)=1
\ee
\end{lemma}
The limit   \eqref{asymptotic_order} and property \eqref{location} of 
$A_\varepsilon$ imply that, for every bounded, local, nondecreasing
function $h:\mathbf{X}\to\R$, 
\be\label{reduce_to_source}
\liminf_{t\to\infty}\Exp h(\eta_t^{\alpha})\geq\liminf_{t\to\infty}\Exp h(\eta_t^{\alpha,t})
\ee
The next main step is to study the asymptotics of the r.h.s. of \eqref{reduce_to_source}.
\begin{lemma}
\label{lemma_asymptotic_source}
 Assume   \eqref{nn_hyp},  and that $\alpha$ is a given environment satisfying  
 \eqref{assumption_afgl}, \eqref{average_afgl}, \eqref{assumption_pt}  and  (H). 
 Let $\beta=\beta(\varepsilon)$ be as in 
 Lemma \ref{lemma_asymptotic_order}. 
Then, for any bounded local  non-decreasing  function $h:\N^\Z\to\R$,
\be\label{asymptotic_source}
\lim_{\varepsilon\to 0}
\liminf_{t\to+\infty}\Exp h(\eta^{\alpha,t}_t)\geq\int_\mathbf{X}h(\eta)d\mu^\alpha_c(\eta)
\ee
\end{lemma}
\begin{remark}\label{recall_epsilon}   The dependence on 
$\varepsilon$ of the l.h.s. of \eqref{asymptotic_source} is hidden 
in the initial configuration $\eta^{\alpha,t}_0$ given by 
\eqref{def_init_source}, where $x_t=\lfloor \beta(\varepsilon) t \rfloor$.
\end{remark}
The combination of \eqref{reduce_to_source} and \eqref{asymptotic_source} 
implies Proposition \ref{prop_main_work}.
\\ \\
We now give the main lines of the proof of Lemma \ref{lemma_asymptotic_order},
 parts of which will be completed in the next sections.
Next, to conclude the outline, we will explain the main idea 
for the proof of Lemma \ref{lemma_asymptotic_source}, which will be 
carried out at the end of Subsection \ref{subsec_loc_eq}.\\
\begin{proof}{Lemma}{lemma_asymptotic_order}
 We rely on the following interface property of nearest-neighbour attractive  systems, for which, 
more generally (see e.g. \cite[Lemma 4.7]{lig} or \cite[Lemma 6.5]{rez}),
the number of sign changes between the difference of two coupled configurations 
in  nearest-neighbour attractive systems is$  $ a nonincreasing function of time. For self-containedness
a proof is given in Appendix \ref{app_lemmas}. 
\begin{lemma}
\label{lemma_interface}
 Let $(\zeta^\alpha_s)_{s\geq 0}$ and $(\varpi^\alpha_s)_{s\geq 0}$ be
 two processes coupled via the Harris system \eqref{def_omega}. Assume 
 there exists a (possibly random) location
$x_0$ such that $\zeta_0(x)\leq \varpi_0(x)$ for $x\leq x_0$, and 
$\zeta_0(x)\geq\varpi_0(x)$ for $x>x_0$. Then there exists a 
piecewise constant c\`adl\`ag $\Z$-valued process $(x_s^\alpha)_{s\geq 0}$, 
with nearest-neighbour jumps, such that $x_0^\alpha=x_0$, and
for all $s\geq 0$, $\zeta_s(x)\leq \varpi_s(x)$ for $x\leq x_s^\alpha$,
 and $\zeta_s(x)\geq\varpi_s(x)$ for $x>x_s^\alpha$.
\end{lemma}
We apply this lemma to $\zeta_s^\alpha=\eta_s^\alpha$ and 
$\varpi_s^\alpha=\eta_s^{\alpha,t}$. In this context, we denote 
the location $x_s^\alpha$ of the lemma by $x_s^{\alpha,t}$.
Then, to establish \eqref{asymptotic_order}, it is enough to show that
\be\label{prob_goodside_0}
\lim_{\varepsilon\to 0}\liminf_{t\to\infty}\Prob\left(\left\{
x_t^{\alpha,t}<A_\varepsilon(\alpha)
\right\}\right)=1
\ee
 Let $b\in(\beta,-v_0)$. On the event $\{x_t^{\alpha,t}\geq A_\varepsilon(\alpha)\}$, 
 by Lemma \ref{lemma_interface}, we have
$$\sum_{x=1+\lfloor bt \rfloor}^{A_\varepsilon(\alpha)}
\eta_t^\alpha(x)\leq \sum_{x=1+\lfloor bt \rfloor}^{A_\varepsilon(\alpha)}\eta_t^{\alpha,t}(x),$$
Therefore, to establish \eqref{prob_goodside_0}, it is enough to prove that  
there exist functions $\beta=\beta(\varepsilon)<b=b(\varepsilon)<-v_0$ such that
$\lim_{\varepsilon\to 0}\beta(\varepsilon)=0$, and 
\be\label{prob_goodside}
\lim_{\varepsilon\to 0}\limsup_{t\to+\infty}\Prob_0\otimes\Prob\left(
\sum_{x=1+\lfloor bt \rfloor}^{A_\varepsilon(\alpha)}\eta_t^\alpha(x)
\leq \sum_{x=1+\lfloor bt \rfloor}^{A_\varepsilon(\alpha)}\eta_t^{\alpha,t}(x)
\right)=0
\ee 
(recall Remark \ref{recall_epsilon}). 
The limit \eqref{prob_goodside}  is a consequence of the following lemma. 
\begin{lemma}\label{lemma_total_mass}
\mbox{}\\ 
(i) For every $\varepsilon>0$ and $b<-v_0$, there exists a family 
$(\delta_t)_{t>0}$ of nonnegative random 
variables on $(\Omega_0,\mathcal F_0,\Prob_0)$, such that  
$\lim_{t\to+\infty}\delta_t=0$  in $\Prob_0$-probability, and
\be\label{left_inside}
\limsup_{t\to\infty}
\Exp_0\Exp\left[
\left(
t^{-1}\sum_{x=1+\lfloor bt \rfloor}^{A_\varepsilon(\alpha)}\eta_t^\alpha(x)+b\rho_c
\right)^--\delta_t
\right]^+\leq  2\varepsilon-b[\rho_c-\overline{R}(c-\varepsilon)]
\ee
(ii)  For small enough $\kappa>0$, there exist $b=b_\kappa$ and
 $\beta=\beta_\kappa$ such that $\beta_\kappa<b_\kappa<-v_0$,
  $\lim_{\kappa\to 0}(b_\kappa-\beta_\kappa)=0$, and 
\be\label{upper_bound_for_xi}
\limsup_{t\to+\infty}\Exp_0\Exp\left[
t^{-1}\sum_{x=1+\lfloor b_\kappa t \rfloor}^{0}\eta_t^{\alpha,t}(x)+b_\kappa\rho_c+\kappa
\right]^+=0
\ee 
\end{lemma}
{\em Proof of \eqref{prob_goodside}.}
Let $\varepsilon>0$, and $\delta(\varepsilon)$ denote the r.h.s. of \eqref{left_inside}. 
By definition \eqref{other_def_critical} of $\rho_c$, we have 
$\lim_{\varepsilon\to 0}\delta(\varepsilon)=0$. Let
 $b:=b_\kappa=:b(\varepsilon)$ and $\beta:=\beta_\kappa=:\beta(\varepsilon)$ be given by
 Lemma \ref{lemma_total_mass}{\it (ii)} for $\kappa=2\delta(\varepsilon)^{1/2}$.
We set
$$
S_t:  =  t^{-1}\sum_{x=1+\lfloor bt \rfloor}^{A_\varepsilon(\alpha)}\eta_t^\alpha(x)+b\rho_c,\quad
S'_t  :=  t^{-1}\sum_{x=1+\lfloor bt \rfloor}^{A_\varepsilon(\alpha)}\eta_t^{\alpha,t}(x)+b\rho_c
$$
Then
\begin{eqnarray*}
&&\Prob_0\otimes\Prob\left(
S_t<-\frac{5}{4}\delta(\varepsilon)^{1/2}
\right)\\
& \leq & 
\Prob_0\otimes\Prob\left(
\delta_t>\frac{\delta(\varepsilon)^{1/2}}{4}
\right)+
\Prob_0\otimes\Prob\left(
S_t<-\frac{5}{4}\delta(\varepsilon)^{1/2}; \,
\delta_t>\frac{\delta(\varepsilon)^{1/2}}{4}
\right)\\
& \leq & \Prob_0\otimes\Prob\left(
\delta_t>\frac{\delta(\varepsilon)^{1/2}}{4}
\right)+
\Prob_0\otimes\Prob\left(
\left[S_t^--\delta_t\right]^+\geq\delta(\varepsilon)^{1/2}
\right)
\end{eqnarray*}
Thus, by Markov inequality and \eqref{left_inside},
\be
\limsup_{t\to+\infty}\Prob_0\otimes\Prob\left\{
S_t<-\frac{5}{4}\delta(\varepsilon)^{1/2}
\right\}\leq\delta(\varepsilon)^{1/2}\label{markov_obtain}
\ee
 On the other hand, 
\begin{eqnarray*}
&&\Prob_0\otimes\Prob\left(
S'_t>-\frac{7}{4}\delta(\varepsilon)^{1/2}
\right) \\
&&\qquad \leq  \Prob_0\otimes\Prob\left\{
\left(
t^{-1}\sum_{x=\lfloor bt\rfloor}^0\eta_t^{\alpha,t}(x)+b\rho_c+2\varepsilon
\right)^+\geq\frac{\delta(\varepsilon)^{1/2}}{4}
\right\}
\end{eqnarray*}
Hence, by \eqref{upper_bound_for_xi} and Markov inequality,
\be\label{choice_yields}
\limsup_{t\to+\infty}\Prob_0\otimes\Prob\left\{
S'_t>-\frac{7}{4}\delta(\varepsilon)^{1/2}
\right\}=0
\ee
The result follows from \eqref{markov_obtain} and \eqref{choice_yields}.
\end{proof}
\mbox{}\\ \\
The proof of Lemma \ref{lemma_total_mass}, which will be 
given in Subsection \ref{subsec_super} below, is based on 
the analysis of currents (see Subsection  \ref{subsec_currents}) and 
on Proposition \ref{th_strong_loc_eq}. 
\\ \\
The idea can be sketched as follows. 
To establish \eqref{left_inside}, we consider for $(\eta_t^\alpha)_{t\geq 0}$ 
the incoming current
at site $1+\lfloor bt\rfloor$ and the outgoing current at site $A_\varepsilon(\alpha)$. 
 The latter, by   statement \eqref{upperbound_current_source} of
  Proposition \ref{th_strong_loc_eq}  applied to  $x_t=A_\varepsilon(\alpha)$, 
cannot exceed the maximum current $(p-q)c$ by more than $\varepsilon$ in average,
 because site $A_\varepsilon(\alpha)$ has rate 
at most $c+\varepsilon$. We show that the former cannot be less than $(p-q)ct$
 minus the initial supercritical $\eta^\alpha$-mass between $b$ and $0$.
Therefore, the loss of $\eta^\alpha$-mass on the space interval 
$[1+\lfloor bt \rfloor,0]$ between times $0$ and $t$ cannot exceed the initial 
supercritical mass by more than $\varepsilon$,
which implies that the $\eta^\alpha$-mass at time $t$ is at least $-b\rho_c-\varepsilon$.\\ \\
Let us come back to Lemma \ref{lemma_asymptotic_source}. 
It is a consequence of strong local equilibrium for the semi-infinite process 
near the origin
 (given in
Proposition \ref{th_strong_loc_eq}), and of Lemma \ref{lemma_entropy}
that we now state.
\\ \\
Recall the definition \eqref{def_v0} of $v_0$  and its announced interpretation 
as the speed of a critical front issued by the source.
We define $\lambda^-(v)$ 
as the smallest
maximizer of  $\lambda\mapsto (p-q)\lambda-v\overline{R}(\lambda)$
over $\lambda\in[0,c]$.
Let $\lambda_0$ denote the smallest minimizer of \eqref{def_v0}, or $\lambda_0=c$ 
if the infimum in \eqref{def_v0} is achieved 
only for $\lambda$ tending to $c$, that is under condition (H). 
 Equivalently, $\mathcal R(v):=\overline{R}[\lambda^-(v)]$ is the smallest
maximizer of  $\rho\mapsto f(\rho)-v\rho$
over $\rho\in[0,\rho_c]$. We will see in Subsection \ref{subsec_loc_eq} that 
$\mathcal R(.)$ is the hydrodynamic density profile generated by the source.
\begin{lemma}\label{lemma_entropy} 
\mbox{}\\ \\
(i) For every $v<v_0$, $\lambda^-(v)=c$.\\ \\
(ii) 
For every $v>v_0$, $\lambda^-(v)<\lambda_0$, and $\lim_{v\downarrow v_0}\lambda^-(v)=\lambda_0$.
\end{lemma}
Lemma \ref{lemma_entropy} is proved in \cite{bmrs1}. It shows that $\overline{R}(\lambda_0)$ is the 
density observed right behind the front.
In particular, under assumption (H), this density is $\rho_c$. Therefore, 
by choosing the position of the source close enough to $-v_0$ we can make 
the density of $\eta^{\alpha,t}_t$ in a neighborhood of zero  close to $\rho_c$. 
This is the idea of Lemma \ref{lemma_asymptotic_source}.
\subsection{Hydrodynamics and strong local equilibrium}\label{subsec_loc_eq}
The hydrodynamic behaviour of the disordered zero-range process is expected, 
and in some cases proven, to be given by the entropy solution to a scalar 
conservation  law  of the form
\be\label{conservation_law}
\partial_t \rho(t,x)+\partial_x f[\rho(t,x)]=0
\ee
for a flux function $f$   constructed from the microscopic dynamics 
(see \eqref{def_flux} below). 
Convergence of the disordered zero-range process to the entropy solution of 
\eqref{conservation_law} is proved in \cite{bfl} for subcritical Cauchy data.
For our purpose we need hydrodynamic limit for the process starting with a source,
 which is not considered in \cite{bfl}.
Besides we also need a strong local equilibrium statement. The latter was derived 
for the homogeneous zero-range process with strictly convex flux in \cite{lan}.
However, the method used there relies on translation invariance of the dynamics, 
which fails in the disordered case. The strategy introduced 
in \cite{bgrs}, where shift invariance is restored by considering the joint 
disorder-particle process, is not feasible either. Therefore another approach
is required here.\\ \\
 We now recall how to obtain the flux function $f$ in \eqref{conservation_law}. 
It follows from \eqref{eq:theta-lambda} that 
\be\label{mean_rate}
\forall x\in\Z,\,\alpha\in{\bf A},\,\lambda\in[0,c),\,
\int_\mathbf{X}\alpha(x)g(\eta(x))d\mu^\alpha_\lambda(\eta)
=\int_{\N}g(n)d\theta_\lambda(n)=\lambda
\ee
The quantity
$$
\int_\mathbf{X}[p\alpha(x)g(\eta(x))-q\alpha(x+1)g[\eta(x+1)]d\mu^\alpha_\lambda(\eta)=(p-q)\lambda
$$
is the stationary current under $\mu^\alpha_\lambda$. As a function of 
the mean density  $\rho=\overline{R}(\lambda)$  
(see \eqref{average_afgl}--\eqref{assumption_pt}),  
the current  can be written 
\be\label{def_flux}
f(\rho):=(p-q)\overline{R}^{-1}(\rho)
\ee
 Its following basic properties are established in  \cite{bmrs1}. 
\begin{lemma}\label{lemma_properties_flux}
 The functions $\overline{R}$  and
$f$ are increasing and $C^\infty$, respectively from  $[0,(p-q)c]$ to $[0,\rho_c]$ 
and from $[0,\rho_c]$ to $[0,(p-q)c]$.   Besides, $\overline{R}$ 
is strictly convex if $R$ is strictly convex.
\end{lemma}
 Recall that before stating Lemma \ref{lemma_entropy}, 
 we defined $\lambda^-(v)$ 
as the smallest
maximizer of  $\lambda\mapsto (p-q)\lambda-v\overline{R}(\lambda)$
over $\lambda\in[0,c]$, and  $\mathcal R(v):=\overline{R}[\lambda^-(v)]$ as the smallest
maximizer of  $\rho\mapsto f(\rho)-v\rho$
over $\rho\in[0,\rho_c]$.
We also define  
the Legendre transform of the current   
\be\label{lagrangian}
f^*(v):=
\sup_{\rho\in[0,\rho_c]}[f(\rho)-v\rho]
=\sup_{\lambda\in[0,c]}[(p-q)\lambda-v\overline{R}(\lambda)]
\ee
 From standard convex analysis (\cite{roc}), we have that 
\be\label{convex_anal}
\mathcal R(v)=-(f^*)'(v+)=(\hat{f}')^{-1}(v+)
\ee
where $\hat{f}:=f^{**}$ is the concave envelope of $f$, defined by 
\be\label{fdoublestar}
\hat{f}(\rho):=f^{**}(\rho):=\inf_{v\in\R}[\rho v+f^*(v)]=\inf_{v\geq 0}[\rho v+f^*(v)]
\ee
The  last  equality follows from the fact that $f$ is nondecreasing. 
Indeed, in this case, \eqref{lagrangian} implies
that for $v\leq 0$,
$$
f^*(v)=f(\rho_c)-v\rho_c=c-v\rho_c
$$
and plugging this into \eqref{fdoublestar} shows that the infimum can be restricted to $v\geq 0$. 
In \eqref{convex_anal}, $(f^{*})'(v+)$ denotes the right-hand derivative 
of the convex function $f^*$, and $(\hat{f}')^{-1}$ the generalized inverse
of the monotone (but not necessarily strictly monotone) function $\hat{f}'$: 
recall that if $\psi:[0,\rho_c]\to[0,c]$ is a nonincreasing function, 
its generalized inverse $\psi^{-1}(v)$ is any function such that 
$\psi^{-1}(v+)\leq \psi^{-1}(v)\leq \psi^{-1}(v-)$, where 
\begin{eqnarray}
\psi^{-1}(v-) & := & \sup\{\rho\in[0,\rho_c]:\,\psi(\rho)>v\}\label{inverse_minus}\\
\psi^{-1}(v+) & := & \inf\{\rho\in[0,\rho_c]:\,\psi(\rho)<v\}\label{inverse_plus}
\end{eqnarray}
It follows from \eqref{convex_anal} that $\mathcal R$ is a nonincreasing and right-continuous function.
\begin{proposition}\label{th_strong_loc_eq} 
Assume $x_t$ in \eqref{def_init_source} is such that $\beta:=\lim_{t\to+\infty}t^{-1}x_t$ 
exists  and is negative.  
Then statement \eqref{upperbound_current_source} below holds, 
statement \eqref{hdl_source} below holds for all $v>0$, and
statement \eqref{the_first_one} below holds for 
all $v>v_0$ and $h:\N^\Z\to\R$  a bounded local  non-decreasing  function: 
\begin{eqnarray}
\limsup_{t\to\infty}\left\{\Exp\left|
t^{-1}\sum_{x>x_t}\eta^{\alpha,t}_t(x)-(p-q)c
\right| - p[\alpha(x_t)-c]\right\}
& \leq & 0 \label{upperbound_current_source}\\
\lim_{t\to\infty}\Exp
\left\vert
t^{-1}\sum_{x>x_t+\lfloor vt\rfloor}\eta^{\alpha,t}_t(x)-f^*(v)\right\vert
& = & 0 \label{hdl_source}\\
\label{the_first_one}
\liminf_{t\to\infty}\left\{
\Exp h\left(\tau_{\lfloor x_t+vt\rfloor}\eta_t^{\alpha,t}\right)
-\int_\mathbf{X} h(\eta)d\mu_{\lambda^-(v)}^{\tau_{\lfloor x_t+vt\rfloor}\alpha}(\eta)
\right\} & \geq & 0
\end{eqnarray}
\end{proposition}
 \begin{remark}\label{remark_hdl}
Statements \eqref{hdl_source} and \eqref{the_first_one} are indeed a hydrodynamic limit 
and local equilibrium statement under hyperbolic time scaling, formulated at macroscopic 
time $1$ and macroscopic position $v$. One should think of $t\to+\infty$ as the scaling 
parameter (multiplied here by the macroscopic time $1$ to obtain the microscopic time $t$). 
The reduced description at time $1$ is sufficient because the hydrodynamic limit is self-similar
 (see \cite{bgrs} for details).  
The usual  form of  hydrodynamic limit and local equilibrium statements 
at macroscopic time-space location 
$(s,u)$  would be
\begin{eqnarray*}
\lim_{t\to\infty}\Exp
\left\vert
t^{-1}\sum_{x>x_{t}+\lfloor tu\rfloor}\eta^{\alpha,t}_{ts}(x)-sf^*\left(\frac{u}{s}\right)
\right\vert
& = & 0 \\
\liminf_{t\to\infty}\left\{
\Exp h\left(\tau_{\lfloor x_{t}+tu\rfloor}\eta_{ts}^{\alpha,t}\right)
-\int_\mathbf{X} h(\eta)d\mu_{\lambda^-(u/s)}^{\tau_{\lfloor x_t+tu\rfloor}\alpha}(\eta)
\right\} & \geq & 0
\end{eqnarray*}
\end{remark}
\begin{remark}\label{remark_bilateral}
If only the negative half of \eqref{average_afgl} is assumed, 
statements \eqref{hdl_source} and \eqref{the_first_one} still hold 
for $v<-\beta$, and the former can be extended to $v=-\beta$. 
\end{remark}
Statements \eqref{upperbound_current_source}--\eqref{hdl_source}  
deal respectively with the current across the source and 
hydrodynamics away from it. They will be needed to prove \eqref{asymptotic_order}. 
Statement \eqref{the_first_one} is a strong local equilibrium statement
required to prove \eqref{asymptotic_source}. 
 A heuristic explanation for the values $f^*(v)$ and $\lambda^-(v)$ 
 in \eqref{hdl_source}--\eqref{the_first_one} can be found in \cite{bmrs1}. \\ \\
The proof of Proposition \ref{th_strong_loc_eq}, carried out in 
Section \ref{sec_loc_eq}, uses material from the next subsection. \\ \\
With Proposition \ref{th_strong_loc_eq} and Lemma \ref{lemma_entropy},
we are ready for the  proof of Lemma \ref{lemma_asymptotic_source}.\\
\begin{proof}{Lemma}{lemma_asymptotic_source}
By \eqref{the_first_one} with $v=-\beta$, 
$$
\liminf_{t\to+\infty}\Exp h(\eta^{\alpha,t}_t)\geq \int_\mathbf{X}h(\eta)d\mu^\alpha_{\lambda^-(-\beta)}(\eta)
$$
By Lemma \ref{lemma_entropy}, under assumption (H), 
$
\lim_{(-\beta)\downarrow v_0}\lambda^-(-\beta)=c
$.
 The result follows by weak continuity of the measure $\mu^\alpha_\lambda$ 
with respect to $\lambda$  (see \eqref{stoch_inc}). 
\end{proof}
\subsection{Currents}\label{subsec_currents}
Let $x_.=(x_s)_{s\geq 0}$ denote a $\Z$-valued piecewise constant c\`adl\`ag   path such 
that $\vert x_s-x_{s-}\vert\leq 1$ for all $s\geq 0$. In the sequel we will  use  
paths $(x_.)$ independent of the Harris system used for the particle dynamics,
hence we may assume that $x_.$ has no jump time in common with the latter.  
We denote by $\Gamma_{x_.}^\alpha(t,\eta)$ the rightward current across the path $x_.$ 
up to time $t$ in the quenched process  $(\eta_s^\alpha)_{s\geq 0}$  starting from $\eta$ 
in environment $\alpha$, that is  the sum f two contributions. The contribution of particle jumps is
 the number of times a particle jumps from $x_{s-}$ to $x_{s-}+1$ (for $s\le t$), 
 minus the number of times a particle jumps from $x_{s-}+1$ to $x_{s-}$. 
 The contribution of path motion is obtained by summing over jump times
$s$ of the path, a quantity equal to the number of particles at $x_{s-}$ if the jump is to the left, or
 minus  the number of particles at $x_{s-}+1$ if the jump is to the right.
Using   notation \eqref{def_omega}, assumption \eqref{nn_hyp},  and that
 $x_.$ and $\eta_.$ have no jump time in common, this  can be precisely written 
 \begin{eqnarray}
 \Gamma^\alpha_{x_.}(t,\eta) & := &
 \int
 \indicator{
 \left\{u\leq\alpha(x_{s})g[\eta_{s-}^\alpha(x_{s})]\right\}}\indicator{\{s\leq t,\,z=1,\,x=x_{s}\}}
 \omega(ds,dx,du,dz)
 \nonumber\\ 
 &  - & \int
 \indicator{
 \left\{u\leq\alpha(x_{s}+1)g[\eta_{s-}^\alpha(x_{s}+1)]\right\}}\indicator{\{s\leq t,\,z=-1,\,x=x_{s}+1\}}
 \omega(ds,dx,du,dz)\nonumber\\ 
& - & \sum_{0<s\leq t}(x_s-x_{s-})\eta_{s}^\alpha\left[
\max(x_s,x_{s-})
\right]\label{current_harris}
 \end{eqnarray}
 If  $\sum_{x>x_0}\eta(x)<+\infty$,  we also have
 \be\label{current}
 \Gamma^\alpha_{x_.}(t,\eta)=\sum_{x>x_t}\eta_t^\alpha(x)-\sum_{x>x_0}\eta(x)
 \ee
For $x_0\in\Z$, we will write $\Gamma^\alpha_{x_0}$ for the current across the 
fixed site $x_0$; that is, $\Gamma^\alpha_{x_0}(t,\eta):=\Gamma^{\alpha}_{x_.}(t,\eta)$,
where $x_.$ is the constant path defined by $x_t=x_0$ for all $t\geq 0$.\\ \\
 The following results    will be important tools to compare currents.
For a particle configuration $\zeta\in\mathbf{X}$ and a site $x_0\in \Z$, we define
\be\label{def_cfd}
F_{x_0}(x,\zeta):=\left\{
\ba{lll}
\sum_{y=1+x_0}^{x}\zeta(y) & \mbox{if} & x>x_0\\ \\
-\sum_{y=x}^{x_0}\zeta(y) & \mbox{if} & x\leq x_0
\ea
\right.
\ee
Let us couple two processes $(\zeta_t)_{t\geq 0}$ and $(\zeta'_t)_{t\geq 0}$ 
in the usual way through the Harris construction,  with $x_.=(x_s)_{s\geq 0}$ as above.  
\begin{lemma}
\label{lemma_current}
\be\label{current_comparison}\Gamma^\alpha_{x_.}(t,\zeta_0)
-\Gamma^\alpha_{x_.}(t,\zeta'_0)
\geq -\left(0\vee\sup_{x\in\Z}\left[F_{x_0}(x,\zeta_0)-F_{x_0}(x,\zeta'_0)\right]\right)\ee
\end{lemma}
\begin{corollary}\label{corollary_consequence}
For $y\in\Z$, define the configuration 
\be\label{conf-max}
\eta^{*,y}:=(+\infty)\indicator{(-\infty,y]\cap\Z}
\ee
Then, for every $z\in\Z$ such that $y\leq z$ and every $\zeta\in\mathbf{X}$,
\be\label{eq_corollary_consequence}
\Gamma^\alpha_z(t,\zeta)\leq\Gamma^\alpha_z(t,\eta^{*,y})+\indicator{\{y<z\}}\sum_{x=y+1}^z\zeta(x)
\ee
\end{corollary}
 Lemma \ref{lemma_current} and 
Corollary \ref{corollary_consequence} are  proved in  \cite{bmrs1}. 
 From now on, we denote by  $\eta_0$ any configuration satisfying 
 \eqref{cond_init} and \eqref{empty_right}. \\ \\
 The following version of {\em finite propagation property} will be 
used repeatedly in the sequel. See \cite{bmrs1}  for a proof.
\begin{lemma}\label{lemma_finite_prop}
For each $W > 1$, there exists $b =  b(W) >  0$ such that for large enough $t$,
 if $ \eta_0 $ and $\xi _ 0 $ agree on an interval $(x,y)$, then, outside probability $e^{-bt}$,
$$
\eta_s(u) = \xi_s(u) \quad\mbox{for all }0 \leq s \leq t \mbox{ and } u\in (x+Wt,y-Wt)
$$ 
\end{lemma}
Next corollary  to Lemma \ref{lemma_current} follows from the latter  
combined with  
Lemma \ref{lemma_finite_prop}
and  Lemma \ref{lemma_inv_meas} below. 
 It is proved in Appendix \ref{app_lemmas}.
\begin{corollary}
\label{corollary_current}
Assume $\xi_0^{\alpha,c-\varepsilon}\sim\mu^\alpha_{c-\varepsilon}$, with $\varepsilon>0$.
Then,  given $W>1$, 
$$\Gamma^\alpha_{x_0}(t,\eta_0)-\Gamma^\alpha_{x_0}(t,\xi_0^{\alpha,c-\varepsilon})
\geq -\left(0\vee\sup_{x\in[x_0-Wt,x_0+1+Wt]}
\left[F_{x_0}(x,\eta_0)-F_{x_0}(x,\xi_0^{\alpha,,c-\varepsilon})\right]\right)$$
 with $\Prob_0\otimes\Prob$-probability tending to $1$ as $t\to+\infty$.
\end{corollary}
For the current near the origin, we have the following bound, 
where $A_\varepsilon(\alpha)$ was defined in \eqref{def:ell} above.
\begin{lemma}
\label{current_critical_2}
Let $\alpha\in\bf A$ satisfy \eqref{assumption_afgl}.
Then 
$$
\limsup_{t\to\infty}\Exp\left[
t^{-1}\Gamma^\alpha_{A_\varepsilon(\alpha)}(t,\eta_0)
-(p-q)c
\right]^+\leq\varepsilon
$$
\end{lemma}
\begin{proof}{Lemma}{current_critical_2} 
 By Corollary \ref{corollary_consequence} (with $y=z=A_\varepsilon(\alpha)$ and $\zeta=\eta_0$), we have
$\Gamma^\alpha_{A_\varepsilon(\alpha)}(t,\eta_0)
\leq\Gamma^\alpha_{A_\varepsilon(\alpha)}\left(t,\eta_0^{\alpha,t}\right)$,
for $\eta_0^{\alpha,t}$ given by \eqref{def_init_source} with $x_t:=A_\varepsilon(\alpha)$. 
We then apply \eqref{upperbound_current_source} of Proposition \ref{th_strong_loc_eq} 
to the r.h.s. of this inequality.
\end{proof}\\ \\
The following result for the equilibrium current will be important for our purpose.
\begin{lemma}\label{current_critical}
Let $\alpha\in{\bf A}$ satisfy conditions \eqref{assumption_afgl} 
and \eqref{average_afgl}. Assume  $\xi_0^{\alpha,\lambda}\sim\mu^\alpha_\lambda$ with  
$\lambda\in[0,c)$.  Let $(x_t)_{t>0}$ be a $\Z$-valued 
family and assume the limit $\lim_{t\to+\infty}t^{-1}x_t=:\beta$ exists and is negative. Then
$$\lim_{t\to\infty}t^{-1}\Gamma^\alpha_{x_t}\left(t,\xi_0^{\alpha,\lambda}\right)=(p-q)\lambda
\quad\mbox{ in }\,\,L^1(\Prob_0\otimes\Prob)$$
\end{lemma}
The proof of Lemma \ref{current_critical} 
uses the following  lemma, proved in   Appendix \ref{app_lemmas}. 
\begin{lemma}\label{lemma_inv_meas}
For every $\alpha\in{\bf A}$ satisfying \eqref{assumption_afgl} and \eqref{average_afgl}, 
the following limits hold in
 $L^2(\mu_\lambda^\alpha)$,  respectively for every $\lambda\in[0,c)$ 
 in the case of \eqref{ergo_density}, and for every
$\lambda\in[0,c]$ in the case of \eqref{ergo_flux}. 
\begin{eqnarray}
\lim_{n\to+\infty}n^{-1}\sum_{x=-n}^0\eta(x) 
& = & \lim_{n\to+\infty}n^{-1}\sum_{x=0}^n\eta(x)=\overline{R}(\lambda)\label{ergo_density}\\ 
\lim_{n\to+\infty}n^{-1}\sum_{x=-n}^0 \alpha(x)g[\eta(x)] 
& = & \lim_{n\to+\infty}n^{-1}\sum_{x=0}^n \alpha(x)g[\eta(x)]=\lambda\label{ergo_flux}
\end{eqnarray}
\end{lemma}
\begin{proof}{Lemma}{current_critical}
Since part of the following computations will be used later in a 
slightly different context, we begin in some generality by presenting them for a process
$(\zeta^\alpha_t)_{t\geq 0}$ starting from configuration $\zeta_0=\zeta$. 
When required, we shall specialize this computation to our equilibrium process 
 $\zeta_t=\xi^{\alpha,\lambda}_t$. 
We first note that  for $\zeta\in\mathbf{X}$ and $x,y\in\Z$ such that $x<y$,
 analogously to \eqref{current}, 
\be
\label{difference_courants}
\Gamma^\alpha_x(t,\zeta)-\Gamma^\alpha_y(t,\zeta) = \sum_{z=x+1}^y\zeta_t(z)-\sum_{z=x+1}^y\zeta(z)
\ee
 Given $L\in\N,L>1$,  we define a space-averaged current
\be\label{space_averaged_current}
\Gamma^{\alpha,L}_{x}(t,\zeta):=
L^{-1}\sum_{i=0}^{L-1}\Gamma^\alpha_{x+i}(t,\zeta)
\ee
Using \eqref{difference_courants}, we can write
\be\label{error_averaged_current}
\Gamma^\alpha_{x}(t,\zeta)=\Gamma^{\alpha,L}_{x}(t,\zeta)+\Delta^{\alpha,L}_x(t,\zeta)
\ee
where 
\be\label{error_averaged_current_2}
|\Delta^{\alpha,L}_{x}(t,\zeta)|\leq L^{-1}\sum_{i=1}^{L-1}
\left(\sum_{z=x+1}^{x+i}\zeta_t(z)+\sum_{z=x+1}^{x+i}\zeta(z)\right)
=:\widetilde{\Delta}^{\alpha,L}_{x}(t,\zeta)
\ee
 Next, we have by \eqref{current_harris} 
\begin{eqnarray}
t^{-1}\Gamma^{\alpha,L}_{x}(t,\zeta) & = & t^{-1}\int_0^t L^{-1}\sum_{i=0}^{L-1}
p\alpha(x+i)g[\zeta^\alpha_s(x+i)]ds\nonumber\\
& - & t^{-1}\int_0^t L^{-1}\sum_{i=0}^{L-1}q\alpha(x+i+1)g[\zeta^\alpha_s(x+i+1)]ds
+M^{\alpha,L}_{x}(t,\zeta)
\nonumber\\
& = & (p-q)G^{\alpha,L}_{x}(t,\zeta)
+M^{\alpha,L}_{x}(t,\zeta)+O(L^{-1})\label{block_current}
\end{eqnarray}
 where 
\be\label{def_block_g}
G^{\alpha,L}_x(t,\zeta):=t^{-1}\int_0^t L^{-1}\sum_{i=0}^{L-1}
\alpha(x+i)g[\zeta^\alpha_s(x+i)]
ds,
\ee
$O(L^{-1})$ denotes the product of $L^{-1}$ with a uniformly bounded quantity, 
and $M^{\alpha,L}_{x}(t,\zeta)$ is a martingale with quadratic variation
$$
\left<M^{\alpha,L}_{x}(t,\zeta)\right>=O\left(\frac{t}{L^3}\right)
$$
under $\Prob$ for any fixed $\zeta$ (and uniformly over all choices of $\zeta$).
{}From now on, we assume  $\zeta=\xi_0^{\alpha,\lambda}$, $\zeta_t=\xi_t^{\alpha,\lambda}$ and $x=x_t$.
In this case, 
using stationarity of  $\xi^{\alpha,\lambda}_.$,  we have 
\begin{eqnarray}
\Exp_0\Exp[ \widetilde{\Delta}_{x_t}^{\alpha,L}(t,\xi_0^{\alpha,\lambda})]
& = & 2\Exp_0\left\{
L^{-1}\sum_{i=1}^{L-1}\sum_{z=x_t+1}^{x_t+i}\xi^{\alpha,\lambda}_0(z)
\right\}\nonumber\\
& = & 2
L^{-1}\sum_{i=1}^{L-1}\sum_{z=x_t+1}^{x_t+i}R\left[\frac{\lambda}{\alpha(z)}\right]\leq 2L\,R\left(\frac{\lambda}{c}\right)
\label{using_stationarity}
\end{eqnarray}
Now, we take $L=\lfloor\varepsilon t\rfloor$. 
With this choice, by \eqref{using_stationarity},  
$$
\lim_{\varepsilon\to 0}\limsup_{t\to+\infty}\Exp_0\Exp
[t^{-1}\widetilde{\Delta}^{\alpha,\lfloor \varepsilon t\rfloor}_{x_t}(t,\xi_0^{\alpha,\lambda})]=0
$$
By  the triangle  inequality and stationarity of $\xi^{\alpha,\lambda}_.$,
\begin{eqnarray}
&&\Exp_0\Exp\left\vert
G^{\alpha,\lfloor\varepsilon t\rfloor}_x(t,\xi_0^{\alpha,\lambda})-\lambda
\right\vert \nonumber\\
&&\quad\leq  t^{-1}\int_0^t\Exp_0\Exp\left\{
\left\vert
(t\varepsilon)^{-1}\sum_{i=0}^{\lfloor t\varepsilon\rfloor-1}
\alpha(x_t+i)g[\xi^{\alpha,\lambda}_s(x_t+i)]-\lambda
\right\vert
\right\}ds\nonumber\\
&&\quad= 
\Exp_0\Exp\left\{
\left\vert
(t\varepsilon)^{-1}\sum_{i=0}^{\lfloor t\varepsilon\rfloor-1}
\alpha(x_t+i)g[\xi^{\alpha,\lambda}_0(x_t+i)]-\lambda
\right\vert
\right\} 
 \label{upperbound_G}
\end{eqnarray}
 Now, \eqref{ergo_flux} in  Lemma \ref{lemma_inv_meas} and \eqref{upperbound_G} 
imply that, for $\varepsilon>0$ such that $\beta<-\varepsilon$, 
\begin{eqnarray*}
\lim_{t\to+\infty}\Exp_0\Exp\left\vert 
G^{\alpha,\lfloor\varepsilon t\rfloor}_x(t,\xi_0^{\alpha,\lambda})-\lambda\right\vert & = & 0
\end{eqnarray*}
 which concludes the proof. 
\end{proof}
\subsection{Proof of Lemma \ref{lemma_total_mass}}\label{subsec_super}
  This proof is based on the analysis of currents done in 
Subsection  \ref{subsec_currents} (it uses Corollary \ref{corollary_current} 
and Lemma \ref{current_critical_2}, which itself required the first statement 
of Proposition \ref{th_strong_loc_eq}, \eqref{upperbound_current_source}) 
and on the second
statement of Proposition \ref{th_strong_loc_eq}, 
\eqref{hdl_source}. 
\\ \\
 To prepare the proof of Lemma \ref{lemma_total_mass}, we need to relate 
the supercritical mass to the loss of current. This is the object of the following lemma.
 For $\lambda\in[0,c)$, $\xi^{\alpha,\lambda}_0$ denotes a random variable 
 defined on $\Omega_0$ with distribution $\mu^\alpha_\lambda$.
For $t>0,b<0$  and $\varepsilon\in(0,c)$, let 
\begin{eqnarray*}
\delta_b(t,\eta_0,\xi_0^{\alpha,c-\varepsilon}) & := & t^{-1}
\sum_{x=\lfloor bt \rfloor+1}^{0} [\eta_0(x)-\xi_0^{\alpha,c-\varepsilon}(x)] \\
\Delta_b(t,\eta_0,\xi_0^{\alpha,c-\varepsilon}) & := & 
t^{-1}\sup_{\lfloor bt \rfloor<y\leq 0}\sum_{x=\lfloor bt \rfloor+1}^{y} 
[\eta_0(x)-\xi_0^{\alpha,c-\varepsilon}(x)]
\end{eqnarray*}
 In other words, by \eqref{def_cfd},
\be\label{autre-expres-delta}
\Delta_b(t,\eta_0,\xi_0^{\alpha,c-\varepsilon}) = t^{-1}\sup_{\lfloor bt \rfloor<y\leq 0}\left(
F_{\lfloor bt \rfloor}(y,\eta_0)-F_{\lfloor bt \rfloor}(y,\xi_0^{\alpha,c-\varepsilon})
\right)
\ee 
\begin{lemma}
\label{lemma_cdfs} 
 We have the $\Prob_0$-a.s. limits  \\ \\
(i) $\displaystyle\limsup_{t\to\infty}[\Delta_b(t,\eta_0,\xi_0^{\alpha,c-\varepsilon})
-\delta_b(t,\eta_0,\xi_0^{\alpha,c-\varepsilon})]\leq 0$
\\ \\
(ii) 
$\displaystyle\limsup_{t\to\infty}\Big\{
t^{-1}\sup_{x\in[\lfloor bt \rfloor-Wt,\lfloor bt \rfloor+1+Wt]}
\left[F_{\lfloor bt \rfloor}(x,\eta_0)-F_{\lfloor bt \rfloor}(x,\xi_0^{\alpha,c-\varepsilon})\right]$\\
\phantom{blablablablabla}$\displaystyle -\delta_b(t,\eta_0,\xi_0^{\alpha,c-\varepsilon})\Big\}
\leq 0$
\end{lemma} 
\begin{proof}{Lemma}{lemma_cdfs} 
\textit{(i)} Let $\lfloor bt \rfloor<y\leq 0$.
Then 
\begin{eqnarray}
\nonumber t^{-1}\sum_{x=\lfloor bt \rfloor+1}^{y} 
[\eta_0(x)-\xi_0^{\alpha,c-\varepsilon}(x)] 
& = & \delta_b(t,\eta_0,\xi_0^{\alpha,c-\varepsilon})\label{decomp}\\
& - & \nonumber
\left[
t^{-1}\sum_{x=y+1}^{0} \eta_0(x)+t^{-1}\rho_c y
\right]\indicator{\{y<0\}}\\
& + & \left[
t^{-1}\sum_{x=y+1}^{0} \xi_0^{\alpha,c-\varepsilon}(x)+t^{-1}\rho_c y
\right]\indicator{\{y<0\}}\nonumber\\
& =: & \delta_b(t,\eta_0,\xi_0^{\alpha,c-\varepsilon})
-T_1(t,y,\eta_0)+T_1(t,y,\xi_0^{\alpha,c-\varepsilon})\nonumber
\end{eqnarray}
For $\zeta\in\{\eta_0,\xi_0^{\alpha,c-\varepsilon}\}$, let
\begin{eqnarray}\label{second_term}
T_2(t,\zeta) & := & \sup_{\lfloor bt \rfloor<y\leq 0}T_1(t,y,\zeta)
\end{eqnarray}
 We claim that 
\begin{eqnarray}
\label{limit_term2}
\liminf_{t\to\infty} T_2(t,\eta_0) & \geq & 0\\
\label{limit_term3}
\lim_{t\to\infty}T_2(t,\xi_0^{\alpha,c-\varepsilon}) & \leq & 
0 \quad  \Prob_0-a.s.
\end{eqnarray} 
 Indeed, let $y_t$ denote a value of $y$ such that
$$T_1(t,y_t,\eta_0)=T_2(t,\eta_0)$$ 
(such a value exists because the supremum in \eqref{second_term} is over a finite set).
Assume $(t_n)_{n\in\N}$ is a positive sequence such that $\lim_{n\to\infty}t_n=+\infty$.  
Since $\lfloor bt \rfloor<y_t\leq 0$,  there exists a subsequence of $(t_n)_{n\in\N}$ along which
$y_t/t$ has a limit $y\in[b,0]$. If $y<0$, the limit \eqref{limit_term2} along this subsequence 
follows from supercriticality condition \eqref{cond_init} on the initial configuration $\eta_0$. 
If $y=0$, it follows because $T_1(t,y_t,\eta_0)$ is the sum of a nonnegative term and a vanishing term.  
A similar argument  combined with \eqref{ergo_density} of Lemma \ref{lemma_inv_meas}  
establishes \eqref{limit_term3}.  This implies \textit{(i)} since 
$\Delta_b(t,\eta_0,\xi_0^{\alpha,c-\varepsilon})\ge\delta_b(t,\eta_0,\xi_0^{\alpha,c-\varepsilon})$. \\ \\ 
For \textit{(ii)}, since $\eta_0(x)=0$ for $x>0$, we  have 
\be\label{first_half_cdf}
\ba{ll}
& \dsp t^{-1}\sup_{x\in[\lfloor bt \rfloor,\lfloor bt \rfloor+1+Wt]}
\left[F_{\lfloor bt \rfloor}(x,\eta_0)-F_{\lfloor bt \rfloor}
(x,\xi_0^{\alpha,c-\varepsilon})\right] \\ = & 
\dsp t^{-1}\sup_{x\in[\lfloor bt \rfloor,\min(0,\lfloor bt \rfloor+1+Wt)]}
\left[F_{\lfloor bt \rfloor}(x,\eta_0)-F_{\lfloor bt \rfloor}(x,\xi_0^{\alpha,c-\varepsilon})\right]\\
\leq & \dsp \Delta_b(t,\eta_0,\xi_0^{\alpha,c-\varepsilon})
\ea
\ee
 On the other hand, for  $x\in[\lfloor bt \rfloor-Wt,\lfloor bt \rfloor]$,  by \eqref{def_cfd}
\begin{eqnarray}
t^{-1}\left[F_{\lfloor bt \rfloor}(x,\eta_0)-F_{\lfloor bt \rfloor}(x,\xi_0^{\alpha,c-\varepsilon})\right]
& = &-t^{-1}\sum_{y={x}}^{\lfloor bt \rfloor} [\eta_0(y)-\xi_0^{\alpha,c-\varepsilon}(y)]
\nonumber\\ & = & -t^{-1}\sum_{y={x}}^{0} [\eta_0(y)-\xi_0^{\alpha,c-\varepsilon}(y)]\nonumber\\
& + &
\delta_b(t,\eta_0,\xi_0^{\alpha,c-\varepsilon})\label{ontheotherhand}
\end{eqnarray}
 The same argument as in the proof of \textit{(i)}, 
 using assumption \eqref{cond_init} and \eqref{ergo_density} 
 of Lemma \ref{lemma_inv_meas}, shows that $\Prob_0$-a.s.,
$$
\liminf_{t\to\infty}
\sup_{x\in[\lfloor bt \rfloor-Wt,\lfloor bt \rfloor]}
t^{-1}\sum_{y=x}^0[\eta_0(y)-\xi_0^{\alpha,c-\varepsilon}(y)]\geq 0
$$
Together with \eqref{autre-expres-delta},  \eqref{first_half_cdf},
\eqref{ontheotherhand}  and \emph{(i)}, this establishes the lemma.
\end{proof}
\mbox{}\\ \\
We are now ready for the  proof of Lemma \ref{lemma_total_mass}.\\
\begin{proof}{Lemma}{lemma_total_mass}
\mbox{}\\ \\
{\em Proof of \eqref{left_inside}.}
We start from,  using \eqref{difference_courants} for $t$ large enough, 
\begin{eqnarray}
t^{-1}\sum_{x=1+\lfloor bt \rfloor}^{A_\varepsilon(\alpha)}\eta_t^\alpha(x)
 & = & t^{-1}\sum_{x=1+\lfloor bt \rfloor}^{A_\varepsilon(\alpha)}
 \eta_0(x)+t^{-1}\Gamma_{\lfloor bt \rfloor}^\alpha(t,\eta_0)
 -t^{-1}\Gamma_{A_\varepsilon(\alpha)}^\alpha(t,\eta_0)\nonumber\\
 & = & t^{-1}\sum_{x=1+\lfloor bt \rfloor}^{0}[\eta_0(x)-\xi^{\alpha,c-\varepsilon}_0(x)]
 +t^{-1}\sum_{x=1+\lfloor bt \rfloor}^{0}\xi^{\alpha,c-\varepsilon}_0(x)\nonumber\\
& + & t^{-1}\Gamma_{\lfloor bt \rfloor}^\alpha(t,\eta_0)
-t^{-1}\Gamma_{A_\varepsilon(\alpha)}^\alpha(t,\eta_0)\nonumber\\
& - & t^{-1}\sum_{x=1+A_\varepsilon(\alpha)}^0\eta_0(x)\label{insert_remove}
\end{eqnarray}
For $i\in\{2,3,4\}$, we denote by $S_i(t)$ the quantity on the 
$i$-th line of the r.h.s. of \eqref{insert_remove}. Since the sum 
in the deterministic quantity $S_4(t)$ does not depend on $t$, we have 
\be\label{fourth_line}
\lim_{t\to+\infty}S_4(t)=0
\ee
The first term in $S_2(t)$ is  $\delta_b(t,\eta_0,\xi_0^{\alpha,c-\varepsilon})$. 
Using \eqref{ergo_density} in  Lemma \ref{lemma_inv_meas} for the second term, we have
\be\label{first_line}
\lim_{t\to+\infty}\Exp_0\Exp\left|
S_2(t)+b\overline{R}(c-\varepsilon)-\delta_b(t,\eta_0,\xi_0^{\alpha,c-\varepsilon})
\right|= 0
\ee
Next, from Corollary \ref{corollary_current}, Lemma \ref{current_critical} 
and  Lemma \ref{lemma_cdfs}\textit{(ii)},  we obtain 
\begin{eqnarray}
\lim_{t\to+\infty} 
\Exp_0\Exp 
\left\{t^{-1}\Gamma^\alpha_{\lfloor bt \rfloor}(t,\eta_0)
-(p-q)(c-\varepsilon)+\delta_b(t,\eta_0,\xi_0^{\alpha,c-\varepsilon})^+\right\}^-=0 
 \label{incoming}
\end{eqnarray}
 In what follows, we make repeated use of the subadditivity of functions $x\mapsto x^\pm$.
Writing 
\begin{eqnarray*}
S_3(t)+\delta_b(t,\eta_0,\xi_0^{\alpha,c-\varepsilon})^+ &
= &
\left[
t^{-1}\Gamma^\alpha_{\lfloor bt \rfloor}(t,\eta_0)
-(p-q)(c-\varepsilon)+\delta_b(t,\eta_0,\xi_0^{\alpha,c-\varepsilon})^+
\right]\\
& + & \left[
(p-q)c - t^{-1}\Gamma_{A_\varepsilon(\alpha)}^\alpha(t,\eta_0)\right]-(p-q)\varepsilon
\end{eqnarray*}
we obtain
\begin{eqnarray*}
\left[S_3(t)+\delta_b(t,\eta_0,\xi_0^{\alpha,c-\varepsilon})^+\right]^- &
\leq &
\left[
t^{-1}\Gamma^\alpha_{\lfloor bt \rfloor}(t,\eta_0)
-(p-q)(c-\varepsilon)+\delta_b(t,\eta_0,\xi_0^{\alpha,c-\varepsilon})^+
\right]^-\\
& + & \left[
(p-q)c - t^{-1}\Gamma_{A_\varepsilon(\alpha)}^\alpha(t,\eta_0)\right]^-
+ (p-q)\varepsilon
\end{eqnarray*}
and using \eqref{incoming} with Lemma \ref{current_critical_2}, we obtain that 
\be\label{limit_diff_cur}
\limsup_{t\to+\infty}\Exp_0\Exp\left\{
\left[S_3(t)+\delta_b(t,\eta_0,\xi_0^{\alpha,c-\varepsilon})^+\right]^-
\right\}\leq 2\varepsilon
\ee
To conclude we use the decomposition 
\begin{eqnarray*}
S_2(t)+S_3(t)+S_4(t)+b\rho_c & = & \left[
S_2(t)+b\overline{R}(c-\varepsilon)-\delta_b(t,\eta_0,\xi^{\alpha,c-\varepsilon}_0)
\right]\\
& + & \left[S_3(t)+\delta_b(t,\eta_0,\xi^{\alpha,c-\varepsilon}_0)^+\right]+S_4(t)\\
& +  & b[\rho_c-\overline{R}(c-\varepsilon)]-\delta_b(t,\eta_0,\xi^{\alpha,c-\varepsilon}_0)^-
\end{eqnarray*}
which yields the inequality
\begin{eqnarray*}
\left[S_2(t)+S_3(t)+S_4(t)+b\rho_c\right]^- &  \leq & \left[
S_2(t)+b\overline{R}(c-\varepsilon)-\delta_b(t,\eta_0,\xi^{\alpha,c-\varepsilon}_0)
\right]^-\\
& + & \left[S_3(t)+\delta_b(t,\eta_0,\xi^{\alpha,c-\varepsilon}_0)^+\right]^--S_4(t)\\
&  & -b[\rho_c-\overline{R}(c-\varepsilon)]+\delta_b(t,\eta_0,\xi^{\alpha,c-\varepsilon}_0)^-
\end{eqnarray*}
Observe that 
\eqref{limit_term3} 
implies
\be\label{vanish_deltab}
\lim_{t\to+\infty}\delta_b(t,\eta_0,\xi_0^{\alpha,c-\varepsilon})^-=0,\quad\Prob_0\mbox{-a.s.}
\ee
Then, using \eqref{fourth_line}, \eqref{first_line} and \eqref{limit_diff_cur},  
\eqref{left_inside} follows with $\delta_t:=\delta_b(t,\eta_0,\xi_0^{\alpha,c-\varepsilon})^-$.\\ \\
{\em Proof of \eqref{upper_bound_for_xi}.}
By  \eqref{lagrangian} (which implies that $f^*$ is nonincreasing) and \eqref{convex_anal},
$$
f^*(0)-f^*(-b)\leq f^*(0)-f^*(-\beta)=\int_{0}^{-\beta}{\mathcal R}(z)dz
$$
Thus, by equation \eqref{hdl_source} of Proposition \ref{th_strong_loc_eq},
\be
\lim_{t\to\infty}\left[
t^{-1}\sum_{x=\lfloor b t \rfloor}^{0}\eta_t^{\alpha,t}(x)-
\int_{0}^{-\beta}{\mathcal R}(z)dz\right]^+=0\label{source_process}
\ee
in $\Prob_0\otimes\Prob$-probability. 
We can choose  $\beta=\beta_\kappa<b=b_\kappa<-v_0$ so that,
for $\kappa>0$ small enough, 
\be\label{source_process_2}
\int_{0}^{-\beta}{\mathcal R}(z)dz<-b\rho_c-\kappa
\ee
Indeed,  since $\mathcal R$ is equal to $\rho_c$ on  $[0,v_0]$,  the difference
 between the l.h.s. and the r.h.s. of \eqref{source_process_2} can be written
\be\label{difference_integrals}
-\int_{v_0}^{-b}\left[
\rho_c-\mathcal R(x)
\right]dx+
\int_{-b}^{-\beta}
\mathcal R(x)dx+\kappa
\ee
Since the second integrand above is bounded by $\rho_c$, 
one may for instance choose $b-\beta=\varepsilon$ and then
(since $\mathcal R$ is nonincreasing)
\be\label{such_a_choice}
b=-\inf\left\{
y>v_0:\,\int_{v_0}^y \left[
\rho_c-\mathcal R(x)
\right]dx>(1+\rho_c)\kappa
\right\}
\ee
which  implies
$\lim_{\kappa\to 0}b_\kappa=-v_0$. 
\end{proof}
\section{Proof of Proposition \ref{th_strong_loc_eq}}\label{sec_loc_eq}
 We start with the  proof of \eqref{upperbound_current_source}.
\subsection{Proof of \eqref{upperbound_current_source}.}
 To prove \eqref{upperbound_current_source}, we use 
Corollary \ref{corollary_consequence} to compare currents, 
Lemma \ref{current_critical} for equilibrium current,
and variations on the proof of Theorem \ref{proposition_afgl} in Section \ref{sec_afgl}.  \\ \\ 
 Let $\varepsilon\in(0,c)$. We  couple $(\eta_s^{\alpha,t})_{s\geq 0}$ 
 with the stationary process  $(\xi^{\alpha,c-\varepsilon}_t)_{t\geq 0}$, 
 where $\xi^{\alpha,c-\varepsilon}_0\sim\mu^\alpha_{c-\varepsilon}$.
By Corollary  \ref{corollary_consequence} (with $y=z=x_t$ and $\zeta=\xi_0^{\alpha,c}$)
\be\label{compare_bfl}
t^{-1}\Gamma^\alpha_{x_t}\left(t,\eta_0^{\alpha,t}\right)\geq
t^{-1}\Gamma^\alpha_{x_t}\left(t,\xi_0^{\alpha,c-\varepsilon}\right)
\ee
Set
$$S_t:=t^{-1}\Gamma^\alpha_{x_t}\left(t,\eta_0^{\alpha,t}\right)-(p-q)c,$$
Applying  Lemma \ref{current_critical} to the r.h.s. of \eqref{compare_bfl} 
and letting $\varepsilon\to 0$ yields 
\be\label{lowerbound_current_source}
\lim_{t\to+\infty}\Exp\left[
S_t^-
\right]=0
\ee
Next we are going to prove that
\be\label{upperbound_exp_current_source}
\limsup_{t\to+\infty}\Exp\left[
t^{-1}\Gamma^\alpha_{x_t}\left(t,\eta_0^{\alpha,t}\right)-p\alpha(x_t)+qc
\right]\leq 0
\ee
Suppose this is established. 
Then 
by \eqref{upperbound_exp_current_source}, $\limsup_{t\to+\infty}\Exp \{S_t- p[\alpha(x_t)-c]\}\leq 0$.
Since $|S_t|=S_t+2S_t^-$, and \eqref{lowerbound_current_source} holds, 
\eqref{upperbound_current_source} follows.\\ \\
We now prove \eqref{upperbound_exp_current_source}. Let $\delta>0$, 
and denote by $\tilde{\alpha}$ the modified environment
that coincides with $\alpha$ at all sites except at $x_t$, where we set 
$$
\tilde{\alpha}(x_t):=
c-\delta<c<\alpha(x_t)
$$
for $\delta\in(0,c)$.
We couple our source process
$(\eta^{\alpha,t}_s)_{s\geq 0}$ 
with the source process $(\eta^{\tilde{\alpha},t}_s)_{s\geq 0}$ which has a source at 
the same location $x_t$ but environment $\tilde{\alpha}$.
Proposition \ref{monotone_harris} implies that $\eta_s^{\tilde{\alpha},t}\leq\eta_s^{\alpha,t}$ 
for every $s\geq 0$. 
The difference
$$
\sum_{x>x_t}\eta^{\alpha,t}_s(x)-\sum_{x>x_t}\eta^{\tilde{\alpha},t}_s(x)
$$
is not modified by jumps in the bulk. It increases when a particle 
is created at $x_t+1$ for the  original process but not for the modified one.
Note that a particle cannot be removed from the bulk only in the modified process,
because $\eta^{\tilde{\alpha},t}_s(x_t+1) \leq\eta^{\alpha,t}_s(x_t+1)$,
and $g$ is nondecreasing. Thus, in the notation
of \eqref{def_omega}, we have 
\be\label{diff_currents_source}
\sum_{x>x_t}\eta^{\alpha,t}_s(x)-\sum_{x>x_t}\eta^{\tilde{\alpha},t}_s(x)
\leq \omega((0,s]\times\{x_t\}\times(\tilde{\alpha}(x_t),\alpha(x_t)]\times\{1\})
\ee 
As a function of $s$, the r.h.s. of \eqref{diff_currents_source} is a Poisson process in time with intensity
$$
\alpha(x_t)-\tilde{\alpha}(x_t)=
\alpha(x_t)-c+\delta
$$
Hence,
\be\label{excess_mean_current}
\Exp\left[ t^{-1}\Gamma^\alpha_{x_t}(t,\eta_0^{\alpha,t})\right]
\leq \Exp \left[t^{-1}\Gamma^{\tilde{\alpha}}_{x_t}(t,\eta_0^{\tilde{\alpha},t})\right]
+p[\alpha(x_t)-c+\delta]
\ee
We will now show that
\be\label{first_bound_exp_source}
\limsup_{t\to+\infty}\Exp\left[
t^{-1}\Gamma^{\tilde{\alpha}}_{x_t}\left(t,\eta_0^{\tilde{\alpha},t}\right)
\right]\leq (p-q)c
\ee
This, combined with \eqref{excess_mean_current} and $\delta\to 0$ after $t\to+\infty$, 
implies \eqref{upperbound_exp_current_source}.
To prove \eqref{first_bound_exp_source}, 
we use Proposition \ref{prop_jackson} with
$l=x_t$ and $S=\Z\cap(-\infty,l]$.  In this case, the constant function 
$\tilde{\lambda}$ with value $\tilde{\alpha}(l)=c-\delta$ is a solution of
\eqref{jackson_bulk}--\eqref{jackson_boundary}. Since $\tilde{\lambda}(x)<c<\tilde\alpha(x)$ for all $x>l$, 
the measure $\mu^{\tilde{\alpha},S,\tilde{\lambda}(.)}=:\tilde{\mu}$, 
that is the product measure with constant parameter $\tilde{\alpha}(l)=c-\delta$ on 
$\N^{\Z\cap[l+1,+\infty)}$,  is invariant for $L^{\tilde{\alpha},S}$. 
We introduce a stationary process $\xi^{\tilde{\alpha},t}_.$ with generator
$L^{\tilde{\alpha},S}$. We can couple this process to $\eta^{\tilde{\alpha},t}_.$ so that
\be\label{dominating_process}
\eta_s^{\tilde{\alpha},t}(x)\leq\xi_s^{\tilde{\alpha},t}(x),\quad\mbox{for all }s\geq 0\mbox{ and }x>x_t
\ee
We now apply the spatial averaging procedure introduced for the current 
in the proof of Lemma \ref{current_critical}. Recall the quantities
defined in \eqref{space_averaged_current}--\eqref{def_block_g}.
Writing \eqref{block_current} with $x=x_t$  and $\tilde\alpha'$ yields 
\be\label{first_bound_block}
\Exp [t^{-1} \Gamma^{\tilde\alpha,L}_{x_t}(t,\eta_0^{\tilde{\alpha},t})]
=(p-q)\Exp G_{x_t}^{\tilde{\alpha},L}(t,\eta^{\tilde{\alpha},t}_0)+
\Exp [t^{-1} \widetilde{\Delta}_{x_t}^{\tilde\alpha}(t,\eta_0^{\tilde\alpha,t})]+O(L^{-1})
\ee 
Since $G^{\alpha,L}_x(t,\zeta)$ is a nondecreasing (random) function of $\zeta$  and $\alpha$, this implies 
\be\label{bound_block_g}
G^{\tilde{\alpha},L}_{x_t}(t,\eta_0^{\tilde{\alpha},t})\leq
G^{\tilde{\alpha},L}_{x_t}(t,\xi_0^{\tilde{\alpha},t})
\ee
 Since $\xi^{\tilde{\alpha},t}_0\sim\tilde{\mu}$, it follows from \eqref{def_block_g}  that 
$$
\Exp G^{\tilde{\alpha},L}_{x_t}(t,\xi_0^{\tilde{\alpha},t})=c-\delta< c
$$
By Proposition \ref{monotone_harris}, the quantity 
 $\widetilde{\Delta}_x^{\tilde\alpha}(t,\zeta)$ defined in  
\eqref{error_averaged_current_2}  is a nondecreasing function of $\zeta$. 
Hence, by \eqref{dominating_process}, 
\be\label{compare_deltas}
\Exp\widetilde{\Delta}_{x_t}^{\tilde\alpha}(t,\eta^{\tilde{\alpha},t}_0)\leq
\Exp\widetilde{\Delta}_{x_t}^{\tilde\alpha}(t,\xi^{\tilde{\alpha},t}_0)\leq 2L R\left(\frac{c-\delta}{c}\right)
\ee 
where the last equality follows from the same computation as in \eqref{using_stationarity}.
Now we choose $L=L(t)$ in such a way that $L\to+\infty$ and $L/t\to 0$ as $t\to+\infty$.
 Plugging \eqref{compare_deltas} into \eqref{first_bound_block} 
and letting $t\to+\infty$, we obtain \eqref{first_bound_exp_source}.
\subsection{Proof of \eqref{hdl_source}}\label{subsec:proof_hdl}
The proof relies on the microscopic interface property stated in 
Lemma \ref{lemma_interface}.  More precisely, to prove
 \eqref{hdl_source}, we use \eqref{upperbound_current_source} 
 and derive an intermediate result, Proposition \ref{prop_lim_current}. 
 For the latter, we use again results on currents (Lemmas \ref{lemma_current}, 
 \ref{current_critical} and ergodic properties of Lemma \ref{lemma_inv_meas}), 
 and the interface property. 
For 
$\lambda\in[0,c]$, 
we consider the stationary processes  
$\{\xi ^{\alpha,\lambda}_s\}_{s \geq 0}$,
run by the given Harris system, with
initially $\xi ^{\alpha,\lambda}_0\sim\mu^\alpha_\lambda$.
We simultaneously construct   these 
random configurations 
by inversion for all values of $\lambda$.\\ \\
Let us denote by $F_\lambda$ the c.d.f. of the probability measure $\theta_\lambda$ 
defined in \eqref{eq:theta-lambda}, i.e.
$F_\lambda(t):=\theta_\lambda((-\infty,t])$ for every $t\in\R$, and by 
$F_\lambda^{-1}$ the generalized inverse of $F_\lambda$, defined 
as in \eqref{inverse_minus}--\eqref{inverse_plus}.
Let $(V^x)_{x\in\Z}$ be a family of i.i.d. random variables independent of the Harris system, 
such that for every $x\in\Z$, $V^x$ 
is uniformly distributed on $(0,1)$. Then we set 
\be\label{def_inversion}
\xi^{\alpha,\lambda}_0(x):=F^{-1}_{\frac{\lambda}{\alpha(x)}}(V^x)
\ee
It follows from \eqref{def_inversion} 
that if $\lambda\leq\tilde{\lambda}$, then
$\xi^{\alpha,\lambda}_0\leq \xi^{\alpha,\tilde{\lambda}}_0$ a.s. Finally, if for $\Lambda\in [0,c]$ we set  
\be\label{def_riemann_inversion}
\eta^{\alpha,t,\Lambda}_0(x)=\xi_0^{\alpha,\Lambda}(x)\indicator{\{x\leq x_t\}},
\ee
  we have by construction that
\be\label{compare_critstep}
\eta^{\alpha,t}_0\geq\eta^{\alpha,t,\Lambda}_0
\ee
 Further, 
if $\lambda\in[0,c]$, then
$\eta^{\alpha,t}_0(x)\geq\xi^{\alpha,\lambda}_0(x)$ for  $x\leq x_t$, and 
$\eta^{\alpha,t}_0(x)\leq\xi^{\alpha,\lambda}_0(x)$ for $x>x_t$. We may
 therefore consider the interface process $(x_s^{\alpha,\lambda,t})_{s\geq 0}$
given by Lemma \ref{lemma_interface}, such that 
\begin{eqnarray}
\eta^{\alpha,t}_s(x)\geq \xi^{\alpha,\lambda}_s(x) 
& \mbox{if} & x\leq x^{\alpha,\lambda,t}_s\label{special_interface_1}\\
\eta^{\alpha,t}_s(x)\leq \xi^{\alpha,\lambda}_s(x) 
& \mbox{if} & x> x^{\alpha,\lambda,t}_s\label{special_interface_2}
\end{eqnarray}
 Note that $x^{\alpha,\lambda,t}_s> x_t$, because
 $\eta^{\alpha,t}_s(x)=+\infty$ for $x\leq x_t$. Besides, since 
 $\xi_s^{\alpha,\lambda}(x)$  is a nondecreasing function of $\lambda$, 
$x^{\alpha,\lambda,t}_s(x)$  is a nonincreasing function of $\lambda$.\\ \\
The first step towards proving \eqref{hdl_source} is  the following proposition.
\begin{proposition}\label{prop_lim_current}
For every $\Lambda\in [0,c)$ and $v\in(0,-\beta]$, it holds that
\be\label{eq:lim_current}
\lim_{t\to+\infty}\Exp_0\Exp\left|
t^{-1}\sum_{x>\lfloor x_t+vt\rfloor}\eta_t^{\alpha,t,\Lambda}(x)-f^*(v,\Lambda)
\right|=0
\ee
where,  similarly to \eqref{lagrangian}, 
\be\label{retricted_legendre}
f^*(v,\Lambda):=
\sup_{\rho\in[0,\overline{R}(\Lambda)]}[f(\rho)-v\rho]
=\sup_{\lambda\in[0,\Lambda]}[(p-q)\lambda-v\overline{R}(\lambda)]
\ee
\end{proposition}
\begin{proof}{Proposition}{prop_lim_current}
We divide it into a lower bound and an upper bound.\\ \\
{\em Step one.} We prove that, for every $\lambda\in[0,\Lambda]$,
\be\label{stoch_ineq}
\lim_{t\to+\infty}\Exp_0\Exp\left[
t^{-1}\sum_{x>\lfloor x_t+vt\rfloor}\eta_t^{\alpha,t,\Lambda}(x)-(p-q)\lambda+v\overline{R}(\lambda)
\right]^-=0
\ee
 For $s,t\geq 0$, we set
$y_s^t=\lfloor x_t+vs\rfloor$  (where $t$ plays the role of a scaling parameter,
 and $s$ is the actual time variable).
By Lemma \ref{lemma_current}  and \eqref{current}  we have,  for $\lambda\in[0,\Lambda]$, 
\begin{eqnarray}\label{compare_currents_interface}
t^{-1}\sum_{x>\lfloor x_t+vt\rfloor}\eta_t^{\alpha,t,\Lambda}(x)
&=&t^{-1}\Gamma_{y^t_.}(t,\eta^{\alpha,t,\Lambda}_0) \geq  t^{-1}\Gamma_{y^t_.}(t,\xi^{\alpha,\lambda}_0)\\
t^{-1}\Gamma_{y^t_.}(t,\xi^{\alpha,\lambda}_0)
 &=&  t^{-1}\Gamma_{x_t}(t,\xi^{\alpha,\lambda}_0)-t^{-1}\sum_{x=1+x_t}^{y_t^t}\xi_t^{\alpha,\lambda}(x)
\label{decomp_currents_interface}
\end{eqnarray}
 By Lemma \ref{current_critical}, the first term on the r.h.s. of
 \eqref{decomp_currents_interface} converges a.s. to the mean current $(p-q)\lambda$. 
On the other hand, by   \eqref{ergo_density} of  
Lemma \ref{lemma_inv_meas} and stationarity of $\xi^{\alpha,\lambda}_.$, the second term
converges in distribution to $-v\overline{R}(\lambda)$. \\ \\
{\em Step two.}
Let $S_t$ denote the quantity between brackets in \eqref{stoch_ineq}, 
where $\lambda$ is chosen so as to achieve  
$\sup_{\lambda\in[0,\Lambda]}[(p-q)\lambda-v\overline{R}(\lambda)]$  
(which is possible by continuity of $\overline{R}$). 
Since $\vert S_t\vert=2S_t+S_t^-$, to complete the proof of the  proposition,   it is enough to show that 
$$\limsup_{t\to+\infty}\Exp_0\Exp S_t\leq 0,$$
that is,
\be\label{stoch_ineq_2}
\limsup_{t\to+\infty}\Exp_0\Exp\left\{
t^{-1}\sum_{x>\lfloor x_t+vt\rfloor}\eta_t^{\alpha,t,\Lambda}(x)\right\}
\leq f^*(v,\Lambda)
\ee
To this end, it is enough to prove that
\be\label{stoch_ineq_22}
\limsup_{t\to+\infty}\Exp_1\Exp_0\Exp\left\{
t^{-1}\sum_{x>z_t^t}\eta_t^{\alpha,t,\Lambda}(x)\right\}
\leq f^*(v,\Lambda)
\ee
where $(z^t_s)_{s\geq 0}$ is a rate $v$ Poisson process starting from 
$z^t_0:=x_t$, independent of $\eta^{\alpha,t}_.$, and $\Exp_1$ denotes 
expectation with respect to this Poisson process. Indeed, the error between
 the left-hand sides of \eqref{stoch_ineq_2} and \eqref{stoch_ineq_22} is bounded by
$R(\Lambda/c)t^{-1}\Exp_1|y^t_t-z^t_t|$, which vanishes as $t\to+\infty$ by
 the law of large numbers for the Poisson process.\\ \\
To establish \eqref{stoch_ineq_22}, for $l=\varepsilon t$, $m\in\N\setminus\{0\}$ and $L=ml$, 
we consider spatial blocks of length 
$l$, 
$$
B_{l,j}(s)=[z^t_s-L+1+jl,z^t_s-L+1+(j+1)l)\cap\Z,
$$ 
for $j=0,\ldots,m-1$.  
We observe that 
\be\label{observe_eq_loc}
t^{-1}\sum_{x>z^t_t}\eta_t^{\alpha,t,\Lambda}(x)
\leq
 F_L(t,\eta_t^{\alpha,t,\Lambda})
\ee
 where
$$
F_L(s,\eta):=t^{-1}L^{-1}\sum_{i=0}^{L-1}
\sum_{x>z^t_s+i-L+1}\eta(x)
$$
For $s\in[0,t]$, the number of particles to the right of $z_s^t+i-L+1$
 can be modified either by a particle jump from or to this position, or by 
the motion of the Poisson process. Thus
\begin{eqnarray}
\Exp_1\Exp_0\Exp F_L(t,\eta_t^{\alpha,t,\Lambda}) & = & t^{-1}\int_0^t 
\Exp_1\Exp_0\Exp G_L(s,\eta_s^{\alpha,t,\Lambda})ds\nonumber\\
\label{wecanwrite_loc_eq} & +  & \Exp_0\Exp  F_L(0,\eta_0^{\alpha,t,\Lambda})
\end{eqnarray}
where the first term on the r.h.s. is the contribution of particle jumps 
and the next one is the contribution of self-motion, with 
\begin{eqnarray*}
G_L(s,\eta) & := & L^{-1}\sum_{i=0}^{L-1}p\alpha(z^t_s+i-L+1)g[\eta(z^t_s+i-L+1)]\\
& - & L^{-1}\sum_{i=0}^{L-1} q\alpha(z^t_s+i-L+2)g[\eta(z^t_s+i-L+2)]\\
& - & L^{-1}\sum_{i=0}^{L-1} v\eta(z^t_s+i-L+2)
\end{eqnarray*} 
 By   \eqref{ergo_density} of Lemma \ref{lemma_inv_meas}, 
\be\label{small_piece}
\Exp_0\Exp F_L(0,\eta_0^{\alpha,t,\Lambda})
\leq t^{-1}\Exp_0\Exp\sum_{x>x_t-L+1}\eta^{\alpha,t,\Lambda}_0(x)=O(L/t)
\ee 
Since $g$ is bounded, with an error bounded uniformly by a constant
 times $l/L=m^{-1}$, we can replace $G_L(s,\eta)$ by
\be\label{tildeg_loc_eq}
\widetilde{G}_L(s,\eta) := m^{-1}\sum_{j=0}^{m-1}\widetilde{G}_{l,j}(s,\eta)
\ee
where 
\be\label{tildeg_loc_eq_2}
\widetilde{G}_{l,j}(s,\eta):=\widetilde{H}_{l,j}(s,\eta)-\widetilde{K}_{l,j}(s,\eta)
\ee
with 
\be\label{tildeg_loc_eq_22}
\widetilde{H}_{l,j}(s,\eta):=l^{-1}\sum_{x\in B_{l,j}(s)}
(p-q)\alpha(x)g(\eta(x)),\quad
\widetilde{K}_{l,j}(s,\eta):=l^{-1}\sum_{x\in B_{l,j}(s)}
v\eta(x+1)
\ee 
The sequel of the proof develops the following idea. If instead of $\eta^{\alpha,t,\Lambda}_.$, 
we had one of the equilibrium processes $\xi^{\alpha,\lambda}_.$, 
by stationarity, recalling \eqref{mean_density_quenched}, \eqref{average_afgl} and \eqref{mean_rate},
the  expectation of $\widetilde{G}_{l,j}(s,\xi^{\alpha,\lambda}_s)$ (for large $l$) would be close to
 $(p-q)\lambda-v\overline{R}(\lambda)\leq f^*(v,\Lambda)$.
We will show that in some sense, locally, $\eta^{\alpha,t,\Lambda}_.$ is close to $\xi^{\alpha,\lambda}_.$ 
for some random $\lambda$. To this end we use the interface property (Lemma \ref{lemma_interface})
 and a large finite set of values of $\lambda$,
setting $\lambda_k=k\Lambda/n$ for $k=0,\ldots,n$. The process $\eta^{\alpha,t,\Lambda}_.$ has
 one interface with each equilibrium process $\xi^{\alpha,\lambda_k}_.$.
Between two successive interfaces, $\eta^{\alpha,t,\Lambda}_.$ must lie between two consecutive
 equilibrium processes, and thus be close to either one  if $n$ is large.
Besides, if $n\ll L$, this will be true essentially everywhere in our window of size $L$. Eventually,
 limits will be carried out in the following order: $t\to +\infty$, $\varepsilon\to 0$, 
 $m\to+\infty$ and $n\to+\infty$.\\ \\
We now proceed to details of the above idea.
Let $J(s)$ denote the (random) set of indexes $j\in\{0,\ldots,m-1\}$ 
such that the block  $B_{l,j}(s)$  contains none of the  interfaces 
$x^{\alpha,\lambda_k,t}_s$ for $k=0,\ldots,n$. 
Note that $|J(s)|\geq m-n$. If $j\in J(s)$, 
there exists a  random $k=k(s,j)\in\{0,\ldots,n-1\}$ such that   
\be\label{sandwich_boxes}
\xi^{\alpha,\lambda_k}_s (x)  \leq  \eta^{\alpha,t,\Lambda}_s(x)\leq\xi^{\alpha,\lambda_{k+1}}_s (x),\quad
\mbox{for all  }x\in B_{l,j}(s)
\ee 
Note that \eqref{sandwich_boxes} is true even if all interfaces 
lie to the right of $B_{l,j}(s)$. Indeed, in this case, 
$\eta^{\alpha,t,\Lambda}_s(x)\geq\xi_0^{\alpha,\Lambda}(x)$ for all $x\in B_{l,j}(s)$. 
But since (cf. \eqref{def_riemann_inversion}) $\eta_0^{\alpha,t,\Lambda}\leq\xi_0^{\alpha,\Lambda}$, 
by Proposition \ref{monotone_harris}, $\eta_s^{\alpha,t,\Lambda}\leq\xi_s^{\alpha,\Lambda}$. 
Thus, $\eta^{\alpha,t,\Lambda}_s(x)=\xi_s^{\alpha,\Lambda}(x)$ for all $x\in B_{l,j}(s)$.\\ 
Since $\widetilde{H}_{l,j}(s,\eta)$ and $\widetilde{K}_{l,j}(s,\eta)$ are increasing 
functions of $\eta$, for $j\in J(s)$, \eqref{sandwich_boxes} implies 
\begin{eqnarray*}
\widetilde{G}_{l,j}\left(s,\eta^{\alpha,t,\Lambda}_s\right) 
& \leq & \widetilde{H}_{l,j}\left(s,\xi^{\alpha,\lambda_{k(s,j)+1}}_s\right)-
\widetilde{K}_{l,j}\left(s,\xi^{\alpha,\lambda_{k(s,j)}}_s\right)\\
& \leq & \max_{k=0,\ldots,n}\left\{
\widetilde{H}_{l,j}\left(s,\xi^{\alpha,\lambda_{k+1}}_s\right)-
\widetilde{K}_{l,j}\left(s,\xi^{\alpha,\lambda_{k}}_s\right)
\right\}
\end{eqnarray*}
On the other hand, since $\alpha(.)\leq 1$ and $g(.)\leq 1$, we have 
the rough bound $\widetilde{G}_{l,j}(s,\eta)\leq 1$ for any $\eta\in\mathbf{X}$  and $0\le j\le m-1$. 
Thus
\be\label{thusthus}
\widetilde{G}_{L}\left(s,\eta^{\alpha,t,\Lambda}_s  \right)
\leq \widetilde{\mathcal G}_L\left(s,\eta^{\alpha,t,\Lambda}_s\right)
+\frac{n}{m}
\ee
where 
$$
\widetilde{\mathcal G}_L\left(s,\eta^{\alpha,t,\Lambda}_s\right)
:=m^{-1}\sum_{j\in\{0,\ldots,m-1\}\cap J(s)}\max_{k=0,\ldots,n}\left\{
\widetilde{H}_{l,j}\left(s,\xi^{\alpha,\lambda_{k+1}}_s\right)-
\widetilde{K}_{l,j}\left(s,\xi^{\alpha,\lambda_{k}}_s\right)
\right\}
$$ 
So far, we have reduced the problem to proving that (recall $n/m\to 0$)
\be\label{reduced_to}
\limsup_{t\to+\infty}
t^{-1}\int_0^t\Exp_1\Exp_0\Exp\left\{
 \widetilde{\mathcal G}_L(s,\eta^{\alpha,t,\Lambda}_s)-f^*(v,\Lambda)
\right\}ds \leq 0
\ee
Since 
$$
f^*(v,\Lambda)\geq (p-q)\lambda_{k}-v\overline{R}(\lambda_k)
\geq (p-q)\lambda_{k+1}-v\overline{R}(\lambda_k)-\frac{1}{n}
$$
for every $k\in\{0,\ldots,n\}$, an upper bound for 
$\widetilde{\mathcal G}_L(s,\eta^{\alpha,t,\Lambda}_s)-f^*(v,\Lambda)$ is 
\begin{eqnarray*}
 & &  \frac{1}{n}+m^{-1}\sum_{j\in\{0,\ldots,m-1\}\cap J(s)}\max_{k=0,\ldots,n}\left\{
\widetilde{H}_{l,j}\left(s,\xi^{\alpha,\lambda_{k+1}}_s\right)
-\widetilde{K}_{l,j}\left(s,\xi^{\alpha,\lambda_{k}}_s\right)\right.\\
&&
\left.\phantom{\widetilde{K}_{l,j}\left(s,\xi^{\alpha,\lambda_{k}}_s\right)}-\left[
(p-q)\lambda_{k+1}-v\overline{R}(\lambda_k)
\right]
\right\}\\
& \leq & \frac{1}{n}+m^{-1}\sum_{j=0}^{m-1}
\sum_{k=0}^n\left|
\widetilde{H}_{l,j}\left(s,\xi^{\alpha,\lambda_{k+1}}_s\right)
-\widetilde{K}_{l,j}\left(s,\xi^{\alpha,\lambda_{k}}_s\right)\right.\\
&&
\left.\phantom{\widetilde{K}_{l,j}\left(s,\xi^{\alpha,\lambda_{k}}_s\right)}-\left[
(p-q)\lambda_{k+1}-v\overline{R}(\lambda_k)
\right]
\right|\\
& \leq & \frac{1}{n}+m^{-1}\sum_{j=0}^{m-1}
\sum_{k=0}^n\left|
\widetilde{H}_{l,j}\left(s,\xi^{\alpha,\lambda_{k+1}}_s\right)-(p-q)\lambda_{k+1}\right|\\
& + & \phantom{\frac{1}{n}+}
m^{-1}\sum_{j=0}^{m-1}
\sum_{k=0}^n\left|
\widetilde{K}_{l,j}\left(s,\xi^{\alpha,\lambda_{k}}_s\right)-\overline{R}(\lambda_k)\right| 
\end{eqnarray*}
 In the last expression, the various equilibrium processes are decoupled,
 so for each $k$ we may use the stationarity of the
corresponding equilibrium process to compute its expectation.
 Therefore, to establish \eqref{reduced_to},
it is enough to prove that, for every $\varepsilon>0$, $m\in\N\setminus\{0\}$,
 $n\in\N\setminus\{0\}$, $j\in\{0,\ldots,m-1\}$, $k\in\{0,\ldots,n\}$, 
\begin{eqnarray}\label{block_equilibrium_1}
\lim_{t\to +\infty}t^{-1}\int_0^t\Exp_1\Exp_0\Exp\left|
\widetilde{H}_{l,j}\left(s,\xi^{\alpha,\lambda_{k+1}}_0\right)-(p-q)\lambda_{k+1}\right|ds & = & 0 \\
\label{block_equilibrium_2}
\lim_{t\to +\infty}t^{-1}\int_0^t\Exp_1\Exp_0\Exp\left|
\widetilde{K}_{l,j}\left(s,\xi^{\alpha,\lambda_{k}}_0\right)-\overline{R}(\lambda_k)\right|ds & = & 0
\end{eqnarray} 
where $l=\varepsilon t$.  To prove \eqref{block_equilibrium_1},
we make the change of variable $s=tu$  in \eqref{block_equilibrium_1}, which yields  the integral
$$
\int_0^1\Exp_1\Exp_0\Exp\left|
\widetilde{H}_{l,j}\left(tu,\xi^{\alpha,\lambda_{k+1}}_0\right)-(p-q)\lambda_{k+1}\right|du
$$ 
By  \eqref{ergo_flux} of  Lemma \ref{lemma_inv_meas} 
and the fact that $\lim_{t\to+\infty}t^{-1}z^t_{tu}=\beta+vu$ a.s.
 with respect to the law of the Poisson process, we obtain 
 that the integrand vanishes  for every $u\in[0,1]$  as
 $t\to+\infty$, and the result follows from dominated convergence (domination holds because the
 $\Exp_0\Exp$ expectations of all variables $\alpha(x)g[\xi_0^{\alpha,\lambda_{k+1}}(x)]$
 involved in the above integral are all equal to $\lambda_{k+1}$). 
The proof of \eqref{block_equilibrium_2} is similar to that of 
\eqref{block_equilibrium_1}, except that we use \eqref{ergo_density} instead of \eqref{ergo_flux}.
\end{proof}\\ \\
{\em Proof of \eqref{hdl_source}.} Since $\eta^{\alpha,t}_0\geq\eta^{\alpha,t,\Lambda}_0$,  
by attractiveness,
$$
t^{-1}\sum_{x>x_t+\lfloor vt\rfloor}\eta^{\alpha,t}_t(x)
\geq t^{-1}\sum_{x>x_t+\lfloor vt\rfloor}\eta^{\alpha,t,\Lambda}_t(x)
$$
We apply Proposition \ref{prop_lim_current} to the r.h.s. of the above 
inequality and let $\Lambda\uparrow c$. Since 
$\lim_{\Lambda\uparrow c}f^*(v,\Lambda)=f^*(v)$, this yields
the lower bound 
\be\label{lowerbound_hdl_source}
\lim_{t\to\infty}\Exp
\left[
t^{-1}\sum_{x>x_t+\lfloor vt\rfloor}\eta^{\alpha,t}_t(x)-f^*(v)\right]^-
= 0
\ee
To obtain the upper bound
\be
\label{upperbound_hdl_source}
\lim_{t\to\infty}\Exp
\left[
t^{-1}\sum_{x>x_t+\lfloor vt\rfloor}\eta^{\alpha,t}_t(x)-f^*(v)\right]^+
= 0,
\ee 
we couple $\eta^{\alpha,t}_.$ with the process $\tilde{\eta}^{\alpha,t}_.$, 
whose source is located at site
$\tilde{x}_t:=x_t+a_\varepsilon(\tau_{x_t}\alpha)$  (for $a_\varepsilon$ 
defined in \eqref{def:err}),  and with the process $\tilde{\eta}^{\alpha,t,\Lambda}_.$ 
starting from initial configuration
$\tilde{\eta}^{\alpha,t,\Lambda}_0$ defined as in \eqref{def_riemann_inversion} 
(but replacing $x_t$ by $\tilde{x}_t$).
Let $0<w<v$. 
 Notice that assumption \eqref{assumption_afgl} (or equivalently \eqref{assumption_afgl_2}) implies
$$
\lim_{n\to +\infty}n^{-1}a_\varepsilon(\tau_{-n}\alpha)=0
$$
Hence,  for large enough $t$, $\tilde{x}_t+\lfloor wt\rfloor<x_t+\lfloor vt\rfloor$.
 Since $\eta^{\alpha,t}_0\leq\tilde{\eta}^{\alpha,t}_0$, by attractiveness,
\be\label{attractiveness_upper}
t^{-1}\sum_{x>x_t+\lfloor vt\rfloor}\eta^{\alpha,t}_t(x)
\leq t^{-1}\sum_{x>x_t+\lfloor vt\rfloor}\tilde{\eta}^{\alpha,t}_t(x)
\leq t^{-1}\sum_{x>\tilde{x}_t+\lfloor wt\rfloor}\tilde{\eta}^{\alpha,t}_t(x)
\ee
Let $\Lambda\in[0,c)$. Since $\tilde{\eta}^{\alpha,t}_t\geq\tilde{\eta}_t^{\alpha,t,\Lambda}$,
\begin{eqnarray}
t^{-1}\sum_{x>\tilde{x}_t+\lfloor wt\rfloor}\tilde{\eta}^{\alpha,t}_t(x)
 & \leq & t^{-1}\sum_{x>\tilde{x}_t+\lfloor wt\rfloor}\tilde{\eta}^{\alpha,t,\Lambda}_t(x)
\nonumber\\
& + & t^{-1}\sum_{x>\tilde{x}_t}\tilde{\eta}^{\alpha,t}_t(x)
-t^{-1}\sum_{x>\tilde{x}_t}\tilde{\eta}^{\alpha,t,\Lambda}_t(x)\label{source_lambda}
\end{eqnarray}
For the first  term on the r.h.s. of \eqref{source_lambda}, we use 
Proposition \ref{prop_lim_current}. For the second one we use \eqref{upperbound_current_source}. 
For the third one, we can write
$$
t^{-1}\sum_{x>\tilde{x}_t}\tilde{\eta}^{\alpha,t,\Lambda}_t(x)\geq
t^{-1}\sum_{x>\tilde{x}_t+\lfloor ut\rfloor}\tilde{\eta}^{\alpha,t,\Lambda}_t(x)
$$
for an arbitrarily small $u>0$, and apply again Proposition \ref{prop_lim_current} to the above lower bound.
It follows that
$$
\limsup_{t\to+\infty}\Exp\left\{
t^{-1}\sum_{x>\tilde{x}_t+\lfloor wt\rfloor}\tilde{\eta}^{\alpha,t}_t(x)
-\left[
f^*(w,\Lambda)+(p-q)c-f^*(u,\Lambda)
\right]
\right\}^+
\leq \varepsilon
$$
We then let $\varepsilon\to 0$ and $w\uparrow v$, $u\downarrow 0$ and $\Lambda\uparrow c$,
 so that   $f^*(w,\Lambda)-f^*(u,\Lambda)\to f^*(v)-(p-q)c$,  which establishes
\eqref{upperbound_hdl_source}.
\subsection{Proof of \eqref{the_first_one}} \label{subsec:proof_loc_eq}
 The proof of \eqref{the_first_one} relies on 
 \eqref{hdl_source}, Lemma \ref{lemma_entropy} and the technical Lemma \ref{lemma_equicontinuous}. 
Let 
$$
H(x,\lambda):=\int_\mathbf{X}h(\tau_x\eta)d\mu^\alpha_\lambda(\eta)
$$
In Appendix \ref{app_lemmas}, we prove the following. 
\begin{lemma}\label{lemma_equicontinuous}
The family of functions $\{H(x,.):\,x\in\Z\}$  is equicontinuous on any interval $[0,\Lambda]$ with $\Lambda<c$.
\end{lemma}
By  Lemma \ref{lemma_entropy}{\it (ii)}, for $v>v_0$ we have $\lambda^-(v)<c$.
Thus, thanks to Lemma \ref{lemma_equicontinuous}, to prove \eqref{the_first_one}, 
it is enough to prove that, for every
$\lambda<\lambda^-(v)$,
\be\label{the_first_one_ter}
\liminf_{t\rightarrow \infty}  \left\{\Exp_0\Exp h\left(
\tau_{\lfloor( \beta+v) t \rfloor}\eta^{\alpha,t}_{t}
\right)- \int_\mathbf{X}h(\eta)d\mu^{\tau_{[(\beta+v)t]}\alpha}_{\lambda}(\eta)\right\}\geq 0
\ee
 Let $l\in\N$ such that $h(\eta)$ depends only on  $\{\eta(x):\,x\in\{-l,\ldots,l\}\}$. 
 Let $\lambda<\lambda^-(v)$, or equivalently $\rho<\mathcal R(v)$, where $\rho:=\overline{R}(\lambda)$.
By \eqref{convex_anal},  this is also equivalent to 
$$
v<\hat{f}'(\rho)=\frac{p-q}{\widehat{\overline{R}}'(\lambda)}=:v_\lambda
$$
where $\widehat{\overline{R}}$ denotes the convex envelope of $\overline{R}$, 
and the above equality follows from  \eqref{def_flux}.
Let $w\in (v,v_\lambda)$.
We claim that
\be\label{interface_inside}
\lim_{t\to+\infty}\Prob_0\otimes\Prob\left(
\left\{
x_t^{\alpha,\lambda,t}\geq\lfloor (\beta+w)t \rfloor
\right\}
\right)=1
\ee
Indeed, if \eqref{interface_inside} were not true, there would exist a constant  $C>0$
and a sequence $t_n\to +\infty$ such that 
\be\label{interface_to_left}
\Prob_0\otimes\Prob\left(
\left\{
x_{t_n}^{\alpha,\lambda,t_n}<\lfloor (\beta+w)t_n \rfloor
\right\}
\right)>C
\ee
 for all $n\in\N$.
{}From this we can derive a contradiction. Indeed, the event in \eqref{interface_to_left} implies that 
$$
t_n^{-1}\sum_{x=\lfloor (\beta+w) t_n \rfloor}^{\lfloor (\beta+v_\lambda) t_n\rfloor}
[\eta^{\alpha,t_n}_{t_n}(x)-\xi^{\alpha,\lambda}_{t_n}(x)]\leq 0
$$
But by \eqref{hdl_source} of Proposition \ref{th_strong_loc_eq} and \eqref{convex_anal},  
the above expression converges in probability to
$$
\int_{w}^{v_\lambda}[\mathcal R(u)-\rho]du>0
$$
The event in \eqref{interface_inside} implies that, for $t$ large enough, 
the whole interval $[\lfloor vt\rfloor -l,\lfloor vt\rfloor +l]$
lies to the left of $x^{\alpha,\lambda,t}_t$, hence  $\eta_t^{\alpha,t}$ 
dominates $\xi_t^{\alpha,\lambda}$ on this interval. Thus \eqref{interface_inside} implies
\be\label{whole_interval}
\lim_{t\to+\infty}
\Prob_0\otimes\Prob\left(
h(\eta_t^{\alpha,t})\geq h(\xi^{\alpha,\lambda}_t)
\right)=1
\ee
 Since $h$ is bounded, and $\xi^{\alpha,\lambda}_t\sim\mu^\alpha_\lambda$, 
 \eqref{whole_interval} implies \eqref{the_first_one_ter}.
\begin{appendix}
\section{Proofs of Lemmas}
\label{app_lemmas}
\begin{proof}{Lemma}{lemma_interface}
We will define a process $(x_s)_{s \geq 0} $ (with given initial point $x_0$) having the desired 
properties and in addition is piecewise constant, jumping at time $s$ only 
(necessarily a nearest neighbour jump) when 
either  \\
\indent
$\bullet\, s \in T^{x_{s-}}_.$ 
or \\
\indent
$\bullet\, s \in T^{z}_.$  
for some $|x_{s-}-z| = 1$ and the associated potential jump is to site $x_{s-}$. \\

If at such times we can always give a choice of $x_{s}$ preserving 
the required relations, then we are done.

We first observe that if the jump time $s$ results in no particles moving 
then we can keep $x_.$ constant.  Equally if $s$
entails the movement together of particles from processes $\zeta^\alpha_.$ 
and $\varpi^\alpha_.$, then again we can maintain the value of $x_. $ at time $s$. 
 We address the remaining cases: \\
\indent
$\bullet\, s \in T^{x_{s-}}_.$:  
 in this case we are concerned with the motion of a  $\varpi^\alpha_.$ particle
 and no motion of a $\zeta^\alpha_.$ particle. This implies that
 $\varpi^\alpha_{s-} (x_{s-}) > \zeta^\alpha_{s-} (x_{s-})$. In consequence we automatically have that
$$
\forall y \leq x_{s-}  \ \ \varpi^\alpha_{s-} (y) \geq \zeta^\alpha_{s-} (y )  \mbox{ and }
\forall y > x_{s-}+1  \ \ \varpi^\alpha_{s-} (y) \leq \zeta^\alpha_{s-} (y ).
$$
Therefore in this case we take $x_{s}= x_{s-} $ unless 
$\varpi^\alpha_{s} ( x_{s-}+1)> \zeta^\alpha_{s} ( x_{s-}+1) $,
in which case we put $ x_{s}= x_{s-}+1 $. \\
\indent
 $\bullet\, s \in T^{x_{s-}+1}_.$:  
 in this case we are concerned with the motion of a $ \zeta^\alpha_.$ particle 
 and no motion of a $\varpi^\alpha_.$ particle. This implies that 
 $\zeta^\alpha_{s-} (x_{s-}+1) > \varpi^\alpha_{s-} (x_{s-}+1)$. In consequence we automatically have that 
$$
\forall y < x_{s-}  \ \ \varpi^\alpha_{s-} (y) \geq \zeta^\alpha_{s-} (y )  \mbox{ and }
\forall y \geq x_{s-}+1  \ \ \varpi^\alpha_{s-} (y) \leq \zeta^\alpha_{s-} (y ).
$$
So we take $x_{s}= x_{s-} $ if $\varpi^\alpha_{s} ( x_{s-})\geq \zeta^\alpha_{s} ( x_{s-}) $
otherwise $x_.$ jumps to $  x_{s-}-1 $. \\
\indent
 $\bullet\, s \in T^{x_{s-}-1}_.$:  
this is essentially the same as the second case. 
\end{proof}
\mbox{}\\ \\
\begin{proof}{Corollary}{corollary_current}  
Given $x_0 \in \Z,W>1$, define ${\tilde \eta}_0^\alpha, {\tilde \xi}_0^{\alpha,c-\varepsilon}$ as follows:
 for all $z \in [x_0 - Wt, (x_0 + 1) + Wt]$, 
 ${\tilde \eta}_0^\alpha(z) = \eta_0(z),{\tilde \xi}_0^{\alpha,c-\varepsilon}(z)
 = \xi_0^{\alpha,c-\varepsilon}(z)$  
 and  for $z \notin  [x_0 - Wt, (x_0 + 1) + Wt]$, ${\tilde \eta}_0^\alpha(z) 
 = {\tilde \xi}_0^{\alpha,c-\varepsilon}(z)= 0$.
  Then it follows from the finite propagation property 
 (see Lemma \ref{lemma_finite_prop}) that  ${\tilde \eta}_t^\alpha(z) = \eta_t^\alpha(z),
 {\tilde \xi}_t^{\alpha,c-\varepsilon}(z)= \xi_t^{\alpha,c-\varepsilon}(z)$ 
 for $z = x_0, x_0+1$ with $\Prob_0\otimes\Prob$-probability $\geq 1- e^{-bt}$. 
 Since by \eqref{current_harris} the current $\Gamma^\alpha_{x_0}(.,.) $ depends only 
 on the occupation numbers of sites $x_0$ and $x_0 + 1$, it follows that 
 $\Gamma^\alpha_{x_0}(t,{\tilde \eta}_0^\alpha)=\Gamma^\alpha_{x_0}(t,\eta_0)$ and 
 $\Gamma^\alpha_{x_0}(t,{\tilde \xi}_0^{\alpha,c-\varepsilon})
 = \Gamma^\alpha_{x_0}(t,\xi_0^{\alpha,c-\varepsilon}) $. We conclude by  
 applying Lemma \ref{lemma_current} to ${\tilde \eta}_0^\alpha$ 
 and ${\tilde \xi}_0^{\alpha,c-\varepsilon}$.  
\end{proof}
\mbox{}\\ \\
\begin{proof}{Lemma}{lemma_inv_meas}
Let us denote by $\mathcal R_n$ the l.h.s. of \eqref{ergo_density}, 
and by $\mathcal G_n$ the l.h.s. of \eqref{ergo_flux}. By \eqref{mean_rate}, 
\eqref{mean_density} and \eqref{average_afgl}, we have 
\be\label{expectations_R_G}
\int {\mathcal G_n}(\eta)d\mu^\alpha_\lambda(\eta)=\lambda,\quad
\lim_{n\to+\infty}\int {\mathcal R_n}(\eta)d\mu^\alpha_\lambda(\eta)=\overline{R}(\lambda)
\ee 
for every $\lambda\in[0,c]$.
Since $g$ is bounded, and the random variables $\{\eta(x):\,x\in\Z\}$ 
are independent under $\mu^\alpha_\lambda$, the variance of $\mathcal G_n$ is $O(1/n)$ as $n\to+\infty$.
Thus the result for $\mathcal G_n$ follows from the weak law of large numbers in 
$L^2(\mu^\alpha_\lambda)$.  he same argument works for $\mathcal R_n$ in the case $\lambda<c$, 
because  
$$
V(\mathcal R_n)=\frac{1}{n^2}\sum_{x=-n}^0 V\left(
\frac{\lambda}{\alpha(x)}
\right)
$$
where $V(\lambda)$ denotes the variance of $\theta_\lambda$, and
$R(.)$ and $V(.)$  are bounded on any interval bounded away from $1$. 
\end{proof}
\mbox{}\\ \\
\begin{proof}{Lemma}{lemma_equicontinuous}
Let $l,n\in\N$ be such that $h$ depends only on sites $x\in\{-l,\ldots,l\}$.
Since $h$ is a local function, we can write  
$H(x,\lambda)={H_1(x,\lambda)}/{H_2(x,\lambda)}$,
where $H_1(x,.)$ and $H_2(x,.)$ are power series in $\lambda$, whose derivatives 
can be bounded by power series in $\lambda/c$ with coefficients
independent of $x$. This implies that $\displaystyle{(x,\lambda)
\mapsto \frac{\partial H}{\partial\lambda}(x,\lambda)}$ is uniformly bounded on
$\Z\times[0,\Lambda]$ for any $\Lambda\in[0,c)$.
\end{proof}
 \section{Construction of the process}\label{app:cons}
\subsection{Generator construction}\label{app:gen}
In this subsection, we explain why our assumptions on $g$ 
allow the construction of a Feller semigroup on $\bf X$ 
from usual Hille-Yosida theory and the framework of \cite{ligbook}. 
For unbounded functions $g$, the process can only be constructed 
on $\N^{\Z^d}$, and Hille-Yosida theorem cannot be used, see \cite{and} 
and references therein for functions $g$ with at most linear growth, 
and \cite{bs2} for a class of functions $g$ with superlinear growth.\\ \\
 Let $\mathcal T$ denote the set of finite subsets of $\Z^d$. 
 For each $T\in\mathcal T$ and $\eta\in{\bf X}$, let $c_T(\eta,.)$ 
 denote a finite measure on $\overline{\N}^T$. 
In  \cite[Theorem 3.9]{ligbook}  are considered  Markov pregenerators 
defined on  continuous  cylinder functions $f:{\bf X}\to\R$ by
\be\label{generator_liggett}
Lf(\eta):=\sum_{T\in\mathcal T}\int_{\overline{\N}^T}c_T(\eta,d\xi)[f(\eta^\xi)-f(\eta)]
\ee
where, for $\xi\in\overline{\N}^T$,  $\eta^\xi$ denotes 
the particle configuration defined by $\eta^\xi(x)=\eta(x)$ for $x\not\in T$ 
and $\eta(x)=\xi(x)$ for $x\in T$.
The following theorem states sufficient conditions on the mappings $c_T$ 
for \eqref{generator_liggett} to yield a Markov generator. To this end,
 for $u\in\Z^d$ and $x\in\Z^d$,  one defines the quantities
\begin{eqnarray*}
c_T(u) & := & \sup\left\{
||
c_T(\eta_1,.)-c_T(\eta_2,.)||:\,\eta_1(y)=\eta_2(y)\mbox{ for all } y\neq u
\right\}\\
\gamma(x,u) & := & \sum_{T\in\mathcal T:\,T\ni x}c_T(u),\quad \forall x\neq u
\end{eqnarray*}
 where $||.||$ denotes the total variation norm of a measure.
\begin{theorem}[\cite{ligbook}, Theorem 3.9]
Assume that for every $T\in\mathcal T$, the mapping $\eta\mapsto c_T(\eta,.)$ 
is continuous from $\bf X$ to the set of finite measures on $\overline{\N}^T$ 
with respect to the topology of weak convergence. Assume in addition that
\be\label{assumption_liggett_0}
\sup_{x\in\Z^d}\sum_{T\in\mathcal T:\,T\ni x}c_T(\eta,\overline{\N}^T)<+\infty
\ee
\be\label{assumption_liggett}
\sup_{x\in\Z^d}\sum_{u\in\Z^d:\,u\neq x}\gamma(x,u)<+\infty
\ee
Then the closure of \eqref{generator_liggett} is a Markov generator on $\bf X$.
Thus it generates a Feller 
semigroup and defines a Feller process on $\bf X$.
\end{theorem}
With our assumptions, one can deduce the following result.
\begin{corollary}\label{cor_lig}
Assume $g$ is a nondecreasing continuous function from $\overline{\N}$ to $[0,+\infty)$ 
such that $g(0)=0<g(1)$, and $p(.)$ is a probability measure on $\Z^d$. Then the closure 
of \eqref{generator} is a Markov generator on $\bf X$.  Thus it generates a Feller 
semigroup and defines a Feller process on $\bf X$.
\end{corollary}
\begin{proof}{Corollary}{cor_lig}
We may rewrite \eqref{generator} in the form \eqref{generator_liggett} 
by defining  $c_T=c_T^\alpha$  as follows: $c_T^\alpha(\eta,.)=0$ if $|T|\neq 2$, while for 
$T=\{x,y\}$ with $x\neq y$, 
\be\label{def_ct}
c_{\{x,y\}}^\alpha(\eta,.)
= \alpha(x) p(y-x)g[\eta(x)]
\delta_{
\eta^{x,y}_{|T}
}+
 \alpha(y) p(x-y)g[\eta(y)]\delta_{
\eta^{y,x}_{|T}
}
\ee
where $\eta_{|T}$ denotes the restriction of a configuration $\eta$ to sites in $T$.
It follows that if $|T|\neq 2$,  $c_T^\alpha(u)=0$ for any $u\in\Z^d$;  
while for $T=\{x,y\}$ with $x\neq y$, we have 
$$c_T^\alpha(\eta,\overline{\N}^T)\leq [p(y-x)+p(x-y)]g(+\infty)$$
 which implies \eqref{assumption_liggett_0}. Also from \eqref{def_ct}, 
 we have that  $c_T^\alpha(u)=0$  for $|T|\neq 2$ and $u\in\Z$; while 
 for $T=\{x,y\}$,  $c_T^\alpha(u)=0$  if $u\not\in\{x,y\}$, and 
$$
\max[c_T^\alpha(x),c_T^\alpha(y)]\leq 2[\alpha(x)p(y-x)+\alpha(y)p(x-y)]g(+\infty)
$$
Thus, for $x\neq y$, we have 
$$
\gamma^\alpha(x,y)\leq 2[\alpha(x)p(y-x)+\alpha(y)p(x-y)]g(+\infty)
$$
from which \eqref{assumption_liggett} follows.
\end{proof}
\begin{corollary}\label{cor_liggett_finite}
Let $(\eta_t)_{t\geq 0}$ be a Feller process with generator \eqref{generator} 
and initial state $\eta_0\in\N^{\Z^d}$. Then, almost surely with respect 
to the law of the process, we have $\eta_t\in\N^{\Z^d}$ for all $t>0$.
\end{corollary}
\begin{proof}{corollary}{cor_liggett_finite}
By Corollary \ref{cor_lig}, for any continuous cylinder function $f:{\bf X}\to\R$, 
\be\label{exp_markov}
\Exp f(\eta_t^\alpha)=\Exp f(\eta_0^\alpha)+\int_0^t \Exp\left[
L^\alpha f(\eta_s^\alpha)
\right]ds
\ee
We may apply \eqref{exp_markov} to the  continuous function 
$f(\eta)=\min(\eta(x),M)$ for arbitrary $x\in\Z^d$ and $M\in\N$,
Using boundedness of $g$ and letting $M\to+\infty$, we obtain 
$\Exp\eta_t^\alpha(x)<+\infty$ for every $x\in\Z^d$.
\end{proof}
\subsection{Graphical construction and equivalence}\label{app:graphical}
Most of the following sketch is taken and summarized from \cite{sw}, 
where it is substantially generalized to cover the framewok of \cite{ligbook}. 
Although based on similar percolation ideas than \cite{har}, it is somewhat different and more general. \\ \\ 
For $n\in\N\setminus\{0\}$, consider a growing (unoriented) connected graph $(G_t^{\alpha,n})_{t\geq 0}$,
 where $G_t^\alpha=(V_t^\alpha,E_t^\alpha)$,
with vertex set $V_t^\alpha\subset\Z^d$, and edge set $E_t^\alpha\subset\Z^d\times\Z^d$. 
By ``growing'', we mean that it is a nondecreasing function of time with respect to inclusion. 
The dynamics of this graph is defined as follows. 
For every potential jump event $(t,x,u,z)\in\omega$ such that both $x$ and $x+z$ 
lie in $[-n,n]^d$, 
if  $x\in V^{\alpha,n}_{t-}$ and $x+z\not\in V_{t-}^{\alpha,n}$,  
the graph $G_t^{\alpha,n}$ is obtained by adding  vertex $x+z$ and edge 
$\{x,x+z\}$ to $G_{t-}^{\alpha,n}$. Similarly, if  $x+z\in V^{\alpha,n}_{t-}$ 
and $x\not\in V_{t-}^{\alpha,n}$,  the graph $G_t^{\alpha,n}$ is obtained 
by adding  vertex $x$ and edge $\{x,x+z\}$ to $G_{t-}^{\alpha,n}$.
 This graph process is constructed from the truncated Poisson process
$$
\omega^n(dt,dx,du,dz)=\indicator{\{x\in[-n,n]^d\}}\indicator{\{x+z\in[-n,n]^d\}}\omega(dt,dx,du,dz)
$$
whose intensity 
$$\mu^n(dt,dx,du,dz):=dtdx\indicator{[0,1]}(u)du\indicator{[-n,n]^d}(x)\indicator{[-n,n]^d}(x+z)\, p(z)dz$$
is now a finite measure. Thus 
$\omega^n([0,T]\times\Z^d\times(0,1)\times\Z^d)<+\infty$ with probability $1$,
 so that $(G_t^{\alpha,n})_{t\in[0,T]}$ is a Markov jump process with bounded jump rates. 
 The total rate at which the cardinal of $G_t^{\alpha,n}$ may grow 
 if the current state is  $G_{t-}^{\alpha,n}=G:=(V,E)$, is
$$
\sum_{y\in G}\sum_{x\in(\Z^d\cap[-n,n]^d)\setminus G}\left[p_n(x,y)+p_n(y,x)\right]\leq 2|V|
$$
where $p_n(x,y):=p(y-x)\indicator{[-n,n]^d\times[-n,n]^d}(x,y)$.
It follows that, for a given finite initial graph  $G_0:=(V_0,E_0)$, 
\be\label{expected_graph}\Exp|G_t^{\alpha,n}|\leq |V_0|e^{2t}\ee
 The increasing union 
$$G_t^\alpha:=\bigcup_{n\in\N\setminus\{0\}}G_t^{\alpha,n}:=\left(
\bigcup_{n\in\N\setminus\{0\}}V_t^{\alpha,n},
\bigcup_{n\in\N\setminus\{0\}}E_t^{\alpha,n}
\right)
$$
 defines a growing connected graph.
By \eqref{expected_graph} and monotone convergence, we have  
$\Exp|G_t^\alpha|\leq |V_0|e^{2t}$.  Hence, with probability one, 
$G_t^{\alpha}$ is finite for every  $t\in[0,T]$.  It follows that, almost surely,
\be\label{stationary_convergence} 
\forall T>0,\, \exists n_0\in\N\setminus\{0\},\quad
\forall t\in[0,T],\,\forall n\geq n_0,\,G_t^{\alpha}=G_t^{\alpha,n}
\ee
and that  $(G_t^\alpha)_{t\in[0,T]}$  is the growing graph obtained 
by adding vertices and edges as above, but {\em without} the restriction 
that vertices should belong to $[-n,n]^d$.\\ \\
Now, given a terminal time $T$ and a given site $x$, we may construct 
 growing graphs  $(\check{G}^\alpha_t)_{t\in[0,T]}$  and 
$(\check{G}^{\alpha,n}_t)_{t\in[0,T]}$  starting from vertex $x$ 
backwards in time, that is with respect to the time-reversed Poisson process seen from time $T$,
obtained from \eqref{def_omega} by setting
$$
\check{\omega}(dt,dx,du,dz):=\sum_{n\in\N}\delta_{(T-T_n,X_n,U_n,Z_n)}\indicator{\{T_n\leq T\}}
$$
The reversed Poisson process has the same law as the original Poisson process 
on the time interval $[0,T]$.\\ \\
Consider the backward graph process $(\check{G}_t^{\alpha,x})_{t\in[0,T]}$ 
starting from the terminal graph $\check{G}_0^{\alpha,x}=(\{x\},\emptyset)$
with the single vertex $x$ and no edge. We denote by $\check{V}_t^{\alpha,x}$  
the set of vertices of $\check{G}_t^{\alpha,x}$. 
We claim that for every $t\in[0,T]$, if  $(\eta_s^\alpha)_{s\in[T-t,T]}$ 
is a process satisfying \eqref{rule_1}--\eqref{rule_2}, then $\eta^\alpha_T(x)$ 
depends only on the restriction of $\eta_{T-t}^\alpha$ to sites $x\in\check{V}^{\alpha,x}_t$.
To prove this statement,  assume it is true for some time $t\leq T$, and let
$$
t':=\sup\{
\tau>t:\,\check{G}^{\alpha,x}_\tau=\check{G}^{\alpha,x}_t
\}
$$
Then  by \eqref{rule_2} and by definition of the graph dynamics, we have
\be\label{blank_interval}
\eta^\alpha_{T-\tau}(y)=\eta^\alpha_{T-t}(y),\quad
\forall y\in\check{V}^{\alpha,x}_t,\,\tau\in(t,t')
\ee
since on the time interval $(t,t')$, for the reverse Poisson process, 
there occurs no Poisson event connecting a vertex 
$y\in\check{V}^{\alpha,x}_t$ to a vertex $z\not\in\check{V}^{\alpha,x}_t$. 
It follows from 
\eqref{blank_interval} that if two processes satisfying 
\eqref{rule_1}--\eqref{rule_2} coincide at time $T-t'$ at sites
$y\in\check{V}^{\alpha,x}_{t'-}=\check{V}^{\alpha,x}_{t}$, they will coincide 
at time $T$ at site $x$. Finally, if
the two processes coincide at time $(T-t')_-$ at sites $z\in \check{V}^{\alpha,x}_{t'}$, 
then by \eqref{rule_1} and definition of the graph dynamics, they will coincide 
at time $T-t'$ at sites
$y\in\check{V}^{\alpha,x}_{t'-}$. Since $x$ was arbitrary, this establishes 
uniqueness of a process $(\eta_s^\alpha)_{s\in[0,T]}$ satisfying \eqref{rule_1}--\eqref{rule_2} 
and starting from a given initial configuration $\eta_0$. It shows more precisely that 
the restriction of such a process to the space-time domain
$
\{
(s,y)\in[0,T]\times\Z^d:\, y\in\check{V}^{\alpha,x}_{T-s}
\}
$
can be constructed uniquely by rules \eqref{rule_1}--\eqref{rule_2}, 
where $\omega$ is replaced by the finite measure $\omega^{T,x}$ obtained by removing all potential
events for which the edge $\{x,x+z\}$ does not belong to $\check{E}_T^{\alpha,x}$.
Conversely, for every $t>0$ and $x\in\Z^d$, we may {\em define} $\eta^\alpha_T(x)$ this way.
We claim that the resulting process will satisfy \eqref{rule_1}--\eqref{rule_2}
 with respect to the full measure $\omega$.
Indeed: first, \eqref{rule_2} is satisfied {\em a fortiori} for  $\omega$ because 
it is satisfied for $\omega^{T,x}$, and
the support of the latter is included in that of the former. Next, \eqref{rule_1} 
also holds for $\omega$, because if $(s,y,z,v)\in\omega$, 
then for every $T>s$, we have by construction that $\{y,y+z\}\in\check{E}^{\alpha,x}_{T-s}$, 
and thus $(s,y,z,v)\in\omega^{T,x}$.\\ \\
Consider now the  process  $(\eta^{\alpha,n}_t)_{t\in[0,T]}$  obtained  by  
letting particles jump from $x$ to $y$ only if $(x,y)\in[-n,n]^d$,
that is replacing $\omega$ by $\omega^n$ in \eqref{rule_1}--\eqref{rule_2}. 
This is a jump Markov process with bounded generator similar to \eqref{generator},
 but with $p_n(x,y)$ instead of $p(y-x)$. As above for the unrestricted process, 
 the configuration $\eta^{\alpha,n}_T$ uses only edges from the graph $\check{G}_T^{\alpha,n}$. 
 {}From \eqref{stationary_convergence}, we have  a random $n_0$ such that 
 $\eta_t^{\alpha,n}=\eta_t^\alpha$ for $n\geq n_0$ and $t\in[0,T]$. Hence $\eta^{\alpha,n}_t$ 
 converges a.s. (and thus in law) to $\eta^\alpha_t$ as $n\to+\infty$.\\ \\
On the other hand, \cite[Corollary 3.14]{ligbook} implies that the semigroup 
of the jump processes $(\eta^{\alpha,n}_t)_{t\in[0,T]}$  converges as $n\to+\infty$
to the semigroup of the process with infinitesimal generator 
\eqref{generator}. This shows that the process constructed {\em \`a la Harris} 
coincides with the process defined by \eqref{generator}. \\ \\
\end{appendix}
\mbox{}\\ 
\noindent {\bf Acknowledgments:}
This work was partially supported by ANR-2010-BLAN-0108, 
PICS no. 5470, FSMP,
and grant number 281207 from the
Simons Foundation awarded to K. Ravishankar.
We thank 
Universit\'{e}s Blaise Pascal and Paris Descartes, and EPFL
 for hospitality. 
\end{document}